\documentclass{siamltex}
\usepackage{amsmath,amsfonts,amssymb,mathtools}
\usepackage{url}
\usepackage{latexsym}
\usepackage{graphicx,overpic}
\usepackage{subfigure,caption}
\usepackage{pgfplots}
\pgfplotsset{compat=newest}
\usepackage{color}
\usepackage{booktabs}
\usepackage{caption}
\usepackage{subcaption}
\usepackage{lscape}
\usepackage{array}
\newcolumntype{C}[1]{>{\centering\let\newline\\\arraybackslash\hspace{0pt}}m{#1}}
\usepackage{tikz}

\newcommand{\R}{\mathbb R}

\newcommand{\bs}{\mathbf s}
\newcommand{\bc}{\mathbf c}

\newcommand{\bv}{\mathbf v}
\newcommand{\bw}{\mathbf w}

\newcommand{\cO}{\mathcal O}

\newcommand{\cZ}{\mathcal Z}

\newcommand{\nn}{\mathbb{N}}

\newtheorem{assumption}{Assumption}[section]
\newtheorem{remark}{Remark}[section]
\newtheorem{example}{Example}[section]

\numberwithin{equation}{section}
\begin{document}

\title{Analysis of a  Schwarz-Fourier domain decomposition method}
\author{
Arnold Reusken\thanks{Institut f\"ur Geometrie und Praktische  Mathematik, RWTH-Aachen
University, D-52056 Aachen, Germany (reusken@igpm.rwth-aachen.de).}
}
\maketitle
\begin{abstract}
  The Schwarz domain decomposition method can be used for approximately solving a Laplace equation on a domain formed by the union of two overlapping discs. We consider an inexact variant of this method in which the subproblems on the discs are solved approximately using the projection on a Fourier subspace of the $L^2$ space on the boundary. This model problem is relevant for better understanding of the ddCOSMO solver that is used in computational chemistry. We analyze convergence properties of this Schwarz-Fourier domain decomposition method. The analysis is based on maximum principle arguments. We derive a new variant of the maximum principle and  contraction number bounds in the maximum norm.  
\end{abstract}


\section{Introduction}
Let $\Omega \subset \R^d$, $d=2,3$, be a domain formed by the union of two overlapping open \emph{balls} $B_i$, i.e,  $\Omega =  B_1 \cup B_2$, 
and such that $B_1 \cap B_2 \neq \emptyset$, $\Omega \neq B_i$ for $i=1,2$. For given boundary data  $g$,  the Laplace problem 
\begin{equation} \label{eqLaplace}
  \Delta u = 0 \quad \text{in}~~\Omega,\quad u=g \quad \text{on}~~\partial \Omega,
\end{equation}
can be   approximately solved using  the multiplicative or additive Schwarz domain decomposition method, denoted by Schwarz-DD, which iterates over the balls and solves only Laplace equations on the balls $B_i$. Convergence of this basic method can be proved using established techniques, cf.~\cite{WidlundDD, XuSIAMReview}. The Laplace subproblems on the balls $B_i$ can be solved approximately  in a very efficient way using Fourier series ($d=2$) or spherical harmonics ($d=3$). This basic idea has lead to the development of new numerical simulation methods for a class of implicit solvation models, more precisely, for the COnductor-like Screening MOdel (COSMO)~\cite{Klamt}, which is a particular type of  continuum solvation model. Such models play a fundamental role  in computational chemistry.
In a nutshell, such models account for the mutual polarization between a solvent, described by an infinite continuum, and a charge distribution of a given solute molecule of interest.  We refer to the review articles~\cite{lipparini2016,tomasi2005quantum} for a thorough introduction to continuum solvation models. 
The COSMO approach uses the so-called van der Waals cavity, which models the solute's cavity as a union of balls, each of them centered around an atom. The solution of  a Laplace equation on such a cavity can be approximated using a Schwarz domain decomposition method. This  so-called ddCOSMO solver has been introduced in ~\cite{CaMaSt13} and is further extended in several directions, cf. \cite{Stammetal,Quan2018} for an overview.  This approach has attracted much attention due to its impressive efficiency.
From a numerical analysis point of view this Schwarz-DD technique raises several fundamental questions. For example, for the case of many balls, one is interested in how  the convergence of the Schwarz-DD method depends on the domain geometry, in particular the number of balls and the size of the overlap. This topic is addressed in \cite{CiGa17,Gander2018,CiGa18b,ReSt21}. The results presented in these papers yield contraction number  bounds (in the $H^1$-norm or the $L^\infty$-norm) that show the dependence of the rate of convergence on the number of balls and on the ``globularity'' of the domain.
For the analysis of  Schwarz domain decomposition methods two very different approaches are available, introduced in the seminal work \cite{Lions1988,Lions1989}. The first one is based on a variational formulation and uses the theory of alternating projections in Hilbert spaces cf.~\cite{XuSIAMReview,WidlundDD}. In this approach, used in \cite{CiGa18b,ReSt21}, a Laplace  problem as in \eqref{eqLaplace} is transformed to a Poisson equation with homogeneous  boundary data and the error contraction is measured in the $H^1_0(\Omega)$ norm. The second approach, used in \cite{CiGa17,Gander2018},  is based on the maximum principle and the error contraction is measured in the $L^\infty(\Omega)$ norm.   In all these papers,  convergence issues are studied under the essential assumption  that the subproblems on the balls a \emph{solved exactly}. For the  high efficiency of the ddCOSMO technique, however, it is essential that  the subproblems are solved  \emph{approximately} using  spherical harmonics ($d=3$). In practice relatively low dimensional spherical harmonics spaces turn out to be sufficient for obtaining satisfactory discretization accuracy.  This brings us to a second fundamental question, which motivates  the  topic of this paper: \emph{what are the  convergence properties of the Schwarz-DD if a finite (low) number of Fourier modes ($d=2$) or spherical harmonics ($d=3$) is used to approximate the solution of the Laplace problem on each of the balls.} In the two-dimensional case ($d=2$) we call such a Schwarz-DD method with a finite Fourier approximation on each of the discs  the Schwarz-Fourier-DD method. \\
Numerical experiments show that, also if one uses a small number of  Fourier modes, the method typically converges with a rate that is similar to that of the method with exact solves.   As far as we know the only paper in which the above question is addressed is the recent work \cite{CaHaJhSt24}. In that paper a setting as in \eqref{eqLaplace} with two overlapping discs ($d=2$) is considered and new geometry-dependent estimates for the $L^2$-norm and spectral radius of a Dirichlet to Dirichlet 
operator are derived. This operator maps data on the boundary of a disc to the restriction of its harmonic extension along
circular arcs inside the disc. Using such estimates, convergence results for the Schwarz-Fourier-DD method are derived, e.g. a geometry dependent bound, smaller than one,  for the spectral radius of the error iteration operator and an error contraction   result in the $L^2$-norm for certain geometry configurations (with sufficiently small overlap). In that paper the $L^2$-norm is used, which is natural for the analysis of Fourier approximations.   It is not clear how a convergence analysis of the Schwarz-Fourier-DD method based on the alternating projection technique, using the $H^1(\Omega)$ norm, can be developed. A key difficulty is that in the method one has to deal with discontinuous boundary data that are not in $H^\frac12(\partial B_i)$.  Note that the maximum principle applies to discontinuous boundary data, too.  In this paper we use the technique based on the maximum principle, as introduced in \cite{Lions1989}, to analyze convergence of the Schwarz-Fourier-DD method. \\ 
We study the same basic problem as in  \cite{CaHaJhSt24}, namely the Laplace problem on two overlapping discs \eqref{eqLaplace}. For $B_1$ we take the unit disc. For approximation of the boundary data on $\partial B_i$ we use the real Fourier space of sine and cosine functions with frequency at most $N \in \mathbb{N}$. The dimension of this discretization space is $2N+1$. Solving the Laplace problem on $B_i$ using a finite Fourier sum is represented by $\Delta_i^{-1}P_N^i$, $i=1,2$, with $\Delta_i^{-1}$ the harmonic extension on $B_i$ and $P_N^i$ the Fourier projection of the boundary data. We will show that this operator can be represented as a convolution of the boundary data with a kernel $K_N$ that is an approximation of the Poisson kernel $K$. A first main result that we derive is the positivity of this kernel $K_N$ on a subdomain of $B_1$ formed by all points with distance larger than $\sim \frac{\ln N}{N}$ to the boundary $\partial B_1$ (a similar result holds for $B_2$). This implies a variant of the maximum principle that is a key ingredient in the convergence analysis of the  Schwarz-Fourier-DD method. Based on this we obtain some partial convergence result. In the second part of the paper we consider a Schwarz-Fourier-DD in  which the Fourier projection is replaced by a nodal interpolation in Fourier space, which corresponds to using a discrete Fourier transform. The latter allows a very efficient implementation and we can avoid discontinuities at the intersection points. For this variant we  determine explicit bounds for the maximum norm of the corresponding iteration matrix  hat can be computed numerically.  The results show that, also for small $N$ values, these bounds are very similar to a $L^\infty$-contraction number bound of the Schwarz-DD method applied to the continuous problem. We include results of numerical experiments to  study sharpness of bounds and to illustrate error propagation properties of the methods.

\section{Preliminaries} \label{sectpreliminaries}
We consider the Laplace problem on the unit disc $B=\{\, x \in \R^2\,|\, |x|<1\,\}$ with piecewewise continuous boundary data $g$ on $\partial B$.
The boundary $\partial B$ is parameterized by the angle $\theta \in [0,2\pi)$. We assume that $g$ has at most finitely many points $z_1,\ldots z_m \in [0,2 \pi)$, at which $g$ is discontinuous.  The set of discontinuity locations is denoted by $\cZ=\{z_1,\ldots,z_m\}$.    We assume that the left and right limits at these points exist: $\lim_{\theta \uparrow z_i} g(\theta)=g_i^-$, $\lim_{\theta \downarrow z_i} g(\theta)=g_i^+$. The case $\cZ=\emptyset$ is  allowed. The space of such functions $g$ is denoted by $C_\cZ$. We study the problem
\begin{equation} \label{Laplace} 
\begin{split}
 \Delta u &=0 \quad \text{in}~B \\
   u& =g \quad \text{on}~\partial B.
   \end{split}
\end{equation}
We collect some basics concerning the solution of this problem in terms of Fourier series. The Fourier coefficients of $g$ are given by
\[
 A_n=\frac{1}{\pi} \int_0^{2 \pi} g(\theta) \cos(n\theta) \, d\theta, 
 \quad B_n=\frac{1}{\pi} \int_0^{2 \pi} g(\theta) \sin(n\theta) \, d\theta, \quad n=0,1,2, \ldots,
\]
and the corresponding Fourier series is
\begin{equation} \label{finitesum}
  S_N(\theta)= \tfrac12  A_0 +\sum_{n=1}^N A_n \cos(n\theta) + B_n \sin(n \theta), \quad N \in \nn.
\end{equation}
Because $g \in L^2(\partial B)$ we have $g= \lim_{N \to \infty} S_N$ in $L^2(\partial B)$. Furthermore, the following holds
\begin{equation} \label{prop1} 
\begin{split}
                                \lim_{N \to \infty} S_N(\theta) & = g(\theta) \quad \text{for}~~\theta \in [0,2 \pi] \setminus \cZ,\\
                                \lim_{N \to \infty} S_N(z_i)&= \tfrac12(g_i^- + g_i^+), \quad ~1 \leq i \leq m. 
\end{split}
\end{equation}
On $B$ we use polar coordinates $(\theta,r)$. Below there is some abuse of notation: for $V \subset B$ we use $(\theta,r) \in V$ to denote $r (\cos \theta, \sin \theta) \in V$.  In polar coordinates we have $\Delta f= \frac{\partial^2 f}{\partial r^2}+ \frac{1}{r} \frac{\partial f}{\partial r}+ \frac{1}{r^2} \frac{\partial^2 f}{\partial \theta^2}$, and for $n \in \nn$ we have $\Delta \cos(n \theta) r^n=0$, $\Delta \sin(n \theta) r^n=0$. Using this we obtain that
\begin{equation} \label{solution}
  u(\theta, r):=\tfrac12  A_0 +\sum_{n=1}^\infty r^n \big(A_n \cos(n\theta)+ B_n \sin(n \theta)\big), \quad (\theta,r) \in B,
\end{equation}
solves $\Delta u=0$ in $B$. It is convenient to introduce another representation of $u$ using the Poisson kernel
\begin{equation} \label{Poissonkernel}
  K(\psi,r):= \frac{1-r^2}{1- 2 r \cos \psi +r^2} = 1 + 2 \sum_{n=1}^\infty \cos (n \psi) r^n,  \quad \psi \in [0,2\pi], ~0 \leq r <1. 
\end{equation}
The following holds:
\begin{equation} \label{ResK}
 u(\theta,r)= \frac{1}{2 \pi} \int_0^{2 \pi} K(\theta-\theta',r) g(\theta') \, d\theta'= \frac{1}{2 \pi} \int_0^{2 \pi} K(\theta',r) g(\theta-\theta') \, d\theta'.
\end{equation}
Using this  representation one can show, cf. \cite[Theorem 2.20]{HackbuschElliptic}, $\lim_{r \uparrow 1} u(\theta, r)=g(\theta)$ for all $\theta \in [0,2\pi] \setminus \cZ$. Hence, this $u$ solves the Laplace problem \eqref{Laplace}.

The function $u$ is \emph{not} continuous at  discontinuity points $(z_i,1)$, meaning that $\lim_{(\theta,r) \to (z_i,1)} u(\theta,r)$ does not exist. We give an example that is relevant for the analysis in Section~\ref{sectDDSchwarz} and illustrates the discontinuous behaviour for the special case of a $g$ that is piecewise constant.
\begin{example} \label{exGander}\rm This example is taken from \cite[Theorem 5]{Gander2018}. We take a piecewise constant $g$, with $g=1$ on $[-\theta^\ast,\theta^\ast]$, $0<\theta^\ast<\pi$, $g=0$ on $(\theta^\ast, 2 \pi - \theta^\ast)$. Hence we have two discontinuity locations $z_1=\theta^\ast$, $z_2=2\pi- \theta^\ast$. The solution $u$ of the Laplace equation \eqref{Laplace} is constant  on arcs  $A_{\tilde \theta^\ast}$ of circles parametrized by angles $\tilde \theta^\ast$ between the positive (starting from $M$) $x$-axis and the line that connects $M$  to the discontinuity point $z_1$, cf.  Fig.~\ref{NGFig1}. 
\begin{figure}[ht!]
	\centering \includegraphics[width=0.45\textwidth]{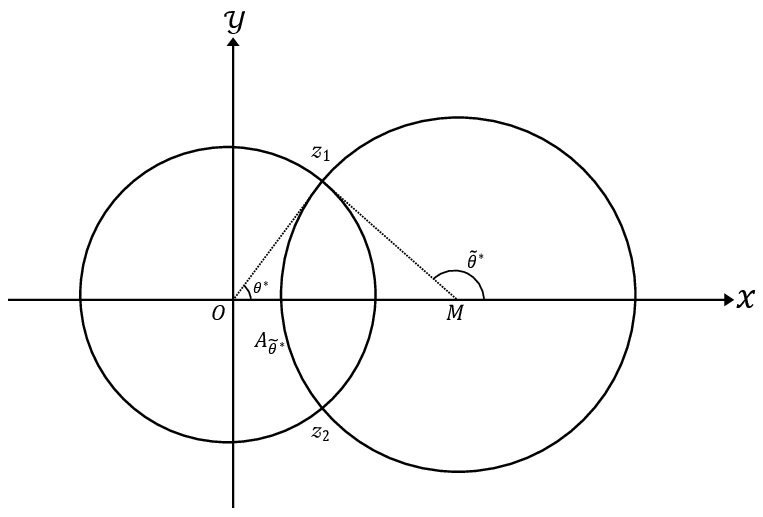}
	\caption{Arcs $A_{\tilde \theta^\ast}$ on which the solution has constant value $u(\theta,r)= \frac{\tilde  \theta^\ast - \theta^\ast}{\pi}$. \label{NGFig1}}
\end{figure}

The \emph{constant} value on the arc $A_{\tilde \theta^\ast}$ is given by $u(\theta,r)= \frac{\tilde  \theta^\ast - \theta^\ast}{\pi}$, $0 < \theta^\ast \leq \tilde \theta^\ast \ < \pi$.
 \end{example}
\ \\

A related result, which makes the discontinuity of $u(\theta,r)$ for $(\theta,r) \to (z_i,1)$  explicit in a more general case, can be derived as follows. Assume, without loss of generality, that $z_1=0$, and $\lim_{\theta \uparrow 2 \pi} g(\theta)=g_0^-$, $\lim_{\theta \downarrow 0} g(\theta)=g_0^+$. Locally, close to $(\theta,r)=(0,1)$ we consider a smooth curve in $B$ that ends at $(0,1)$ in the following sense. For given sufficiently small $\epsilon >0$, assume a $C^1$ function $\theta: (1-\epsilon,1] \to (-\tfrac12 \pi,\tfrac12 \pi) $ with a fixed sign and $\theta(1)=0$. We consider the curve (in polar coordinates) $(\theta(r),r)$, $r \in (1-\epsilon,1]$. 
\begin{lemma} \label{lemdisc}
 Consider a curve $(\theta(r),r)$ as defined above. For the solution $u$ of \eqref{Laplace} on this curve we have:
 \begin{equation} \label{reslemdisc}
  \lim_{r \uparrow 1} u\big(\theta(r),r\big)= \tfrac12(g_0^+ +g_0^-) + (g_0^+-g_0^-){\rm sign}(\theta)  \frac{1}{\pi} \arccos \big( (1+\theta'(1)^2)^{-\frac12}\big).
  \end{equation}
\end{lemma}
\begin{proof}
 A proof is given in the Appendix. 
\end{proof}
\\[2ex]
The result \eqref{reslemdisc} shows that, as expected, the limit value on the curve at the point of discontinuity $(z_1,1)$ is a convex combination of the two limit values $g_0^+$ and $g_0^-$ of the boundary data $g$ on $\partial B$. On a curve with $\theta'(1)=0$, i.e., perpendicular to $\partial B$ at $(z_1,1)$, we have $\lim_{r \uparrow 1} u\big(\theta(r),r\big)= \tfrac12(g_0^+ +g_0^-)$. For the case $|\theta'(1)| \gg 1$, i.e., a curve almost tangential to $\partial B$ at $(z_1,1)$, we have $\lim_{r \uparrow 1} u\big(\theta(r),r\big) \approx  \tfrac12(g_0^+ +g_0^-) + \tfrac12(g_0^+-g_0^-){\rm sign}(\theta)$.

We recall the fundamental maximum principle. For the solution $u= \Delta^{-1}g $ the following holds:
\begin{equation} \label{maximum}
  \min_{\partial B} g \leq u(\theta,r) \leq \max_{\partial B} g \quad \text{for all}~(\theta,r)\in B. 
\end{equation}
One may check that this property is equivalent to the following two properties of the Poisson kernel
\begin{align}
      K(\psi,r) & \geq 0 \quad \text{for all}~(\psi,r) \in B, \label{Max1} \\
      \frac{1}{2 \pi} \int_0^{2 \pi}   K(\psi,r) \, d\psi&  =1 \quad \text{for all}~ 0 \leq r <1. \label{Max2}
\end{align}
\section{Convergence of a Schwarz domain decomposition iteration} \label{sectDDSchwarz}
In this section we recall a classical analysis of the Schwarz method as presented in the seminal paper \cite{Lions1989}. 
We consider the Laplace equation on a domain $\Omega$ that is formed by the union of two overlapping discs. To parameterize the possible geometries we use the angles $\theta^\ast_1$, $\theta^\ast_2$, with $0< \theta^\ast_1 \leq \theta^\ast_2 < \pi$,  as shown in Fig.~\ref{FigGeom}. These two angles uniquely specify the geometry, except for an arbitrary scaling. Without loss of generality we can fix the scaling by taking for the left disc (with angle denoted by $\theta^\ast_1$ in Fig.~\ref{FigGeom}) 
the unit disc with center $(0,0)$, denoted by $B_1=B((0,0);1)$. The intersection points are $(\cos\theta^\ast_1, \sin \theta^\ast_1)$ and $(\cos\theta^\ast_1, -\sin \theta^\ast_1)$.  The other disc has center $(m,0)$ (in Euclidean coordinates), and radius $R$, i.e., $B_2=B((m,0);R)$. 
\begin{figure}[ht!]
	\centering \includegraphics[width=0.45\textwidth]{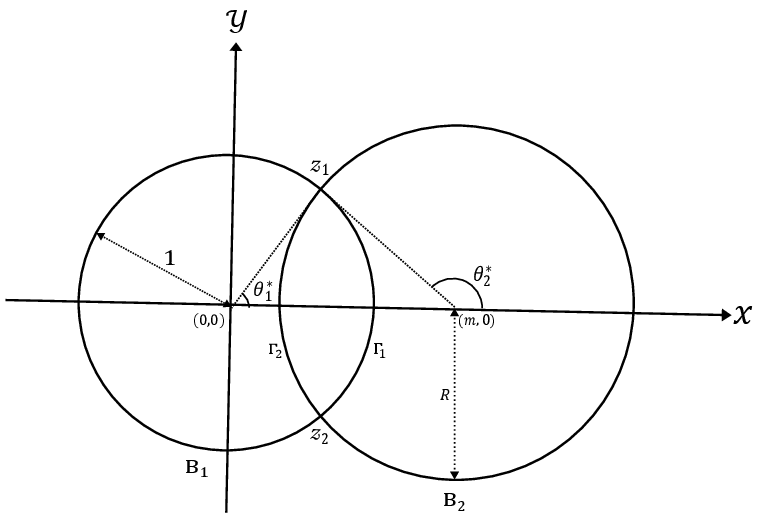}
	\caption{Geometry of two overlapping discs $B_1=B((0,0);1)$ and $B_2=B((m,0);R)$. \label{FigGeom}}
\end{figure}

Given $\theta^\ast_1$ and $\theta^\ast_2$, elementary computations yield that the values of $R$ and $m>0$ are uniquely determined by the relations
\begin{equation} \label{defmA}
  R =\frac{\sin \theta^\ast_1}{\sin \theta^\ast_2}, \quad m=\cos \theta_1^\ast- \frac{\sin \theta_1^\ast}{\tan \theta_2^\ast} \quad \text{if}~\theta_2^\ast \neq \frac{\pi}{2}, \quad m=\cos \theta_1^\ast  \quad \text{if}~\theta_2^\ast = \frac{\pi}{2}.
 \end{equation}
\begin{remark} \label{recontheta}
  \rm If the disc $B_2=B((m,0);R)$ is given one can determine corresponding $\theta^\ast_1, \theta^\ast_2$ as follows. To have intersection points $z_1$ and $z_2$ as in Fig.~\ref{FigGeom} we assume that the intersection points of $B_2$ with the $x$-axis are in $(-1,1)$ and $(1,\infty)$, i.e., $B_2$ is chosen such that $-1 < m-R <1$ and $1<m+R$ (this implies $m>0$). Using the cosine rule we can determine  unique $\theta^\ast_1$ and $\theta^\ast_2$ from
  \begin{equation} \label{formtheta} \begin{split}
                R^2 & =1+m^2 -2 m \cos \theta^\ast_1 \\
                1 & = R^2 + m^2 - 2 R m \cos(\pi - \theta^\ast_2).
                   \end{split}
\end{equation}
\end{remark}

We also use the notation $\Omega:=B_1 \cup B_2$, $\Gamma_1:=\partial B_1 \cap B_2$, $\Gamma_2:=\partial B_2 \cap B_1$, cf. Fig.~\ref{FigGeom}. Hence, $\partial \Omega=(\partial B_1 \setminus \Gamma_1) \cup (\partial B_2 \setminus \Gamma_2)$. We assume boundary data $g$ on $\partial \Omega$ that is continuous and the restriction of $g$ to  a circle section is denoted by $g_i:=g_{|\partial B_i \setminus \Gamma_i}$. 

The classical Schwarz method is an iterative procedure for approximately solving
\begin{equation} \label{Laplace2} \begin{split}
   \Delta u &=0 \quad \text{in}~~\Omega \\
   u&=g \quad \text{on}~~\partial \Omega.
\end{split}
\end{equation}

To simplify the presentation we consider the additive variant of the Schwarz-DD method, cf. Remark~\ref{RemAdditive}. Given $u_i^0 \in C(\overline{B}_i)$, with ${u_i^0}_{|\partial B_i \setminus \Gamma_i}=g_i$, $i=1,2$, we determine for $n\geq 1$:\\
\begin{minipage}{0.5\textwidth}
 \begin{align*}
  \Delta u_1^n &=0\quad \text{in}~B_1\\
   u_1^n&=g_1 \quad \text{on}~\partial B_1 \setminus \Gamma_1 \\
   u_1^n &= u_2^{n-1} \quad \text{on}~\Gamma_1.
 \end{align*}
\end{minipage}
\begin{minipage}{0.49\textwidth}
 \begin{equation} \label{DDschwarzit} \begin{split}
  \Delta u_2^n &=0\quad \text{in}~B_2\\
   u_2^n&=g_2 \quad \text{on}~\partial B_2 \setminus \Gamma_2 \\
   u_2^n &= u_1^{n-1} \quad \text{on}~\Gamma_2.
 \end{split} \end{equation}
\end{minipage}
\\[1ex]
Convergence of this iteration is completely determined by ${u_1^n}_{|\Gamma_2}$ and ${u_2^n}_{|\Gamma_1}$. To put the iteration in a more convenient form we introduce further notation. The solution of the Laplace equation on $B_i$ with piecewise continuous boundary data $(v,w) \in C(\overline{\partial B_i \setminus \Gamma_i}) \times C(\overline{\Gamma}_i)$ is denoted by $\Delta_i^{-1}(v,w)$, $i=1,2$. The function $\Delta_1^{-1}(v,w)$ can be restricted to $\Gamma_2$ and this function has a continuous extension to $\overline{\Gamma}_2$, cf. Lemma~\ref{lemdisc}. This restriction operator $B_1 \to \overline{\Gamma}_2$ is denoted by $R_{\overline{\Gamma}_2}$. Similarly we define $R_{\overline{\Gamma}_1}$. The sequences $(u_i^n)_{n \geq 1}$, $i=1,2$, satisfy
\[
  {u_1^n}_{|\overline{\Gamma}_2}= R_{\overline{\Gamma}_2}\Delta_1^{-1}(g_1,{u_2^{n-1}}_{|\overline{\Gamma}_1}), \qquad 
{u_2^n}_{|\overline{\Gamma}_1}= R_{\overline{\Gamma}_1}\Delta_2^{-1}(g_2,{u_1^{n-1}}_{|\overline{\Gamma}_2}).
  \]
For a more compact notation we also introduce the Dirichlet to Dirichlet map $L_1: \, C(\overline{\Gamma}_1) \to C(\overline{\Gamma}_2)$, 
\begin{equation} \label{defL1} L_1v:=R_{\overline{\Gamma}_2}\Delta_1^{-1}(0,v),
\end{equation}
 and similarly $L_2:\, C(\overline{\Gamma}_2) \to C(\overline{\Gamma}_1)$. We thus obtain
\begin{equation} \label{iteration}
 \begin{pmatrix}
  {u_1^n}_{|\overline{\Gamma}_2}\\
  {u_2^n}_{|\overline{\Gamma}_1}
 \end{pmatrix} = 
 \begin{pmatrix}
   0 & L_1 \\ L_2 & 0
  \end{pmatrix}
  \begin{pmatrix}
  {u_1^{n-1}}_{|\overline{\Gamma}_2}\\
  {u_2^{n-1}}_{|\overline{\Gamma}_1}
 \end{pmatrix}
   +  \begin{pmatrix}
       R_{\overline{\Gamma}_2}\Delta_1^{-1}(g_1,0)\\
       R_{\overline{\Gamma}_1}\Delta_2^{-1}(g_2,0)
      \end{pmatrix}.
\end{equation}
The operator $L=\begin{pmatrix}
   0 & L_1 \\ L_2 & 0
  \end{pmatrix}$ determines the convergence properties of the method. Below we use the maximum norm on $C(\overline{\Gamma}_2) \times C(\overline{\Gamma}_1)$ and  derive $\|L\|_\infty < 1$. Thus we have a contraction and a unique fixed point $\big({u_1^\infty}_{|\overline{\Gamma}_2},
  {u_2^\infty}_{|\overline{\Gamma}_1}\big)$ in the iteration \eqref{iteration}. One easily checks that $u_1:=\Delta_1^{-1}(g_1,{u_2^\infty}_{|\overline{\Gamma}_1})$ is harmonic on $B_1$ and $u_2:=\Delta_2^{-1}(g_2,{u_1^\infty}_{|\overline{\Gamma}_2})$ is harmonic on $B_2$ and $u_1=u_2$ on $B_1\cap B_2$. Hence this pair solves the Laplace problem~\eqref{Laplace2}.
For the contraction number of \eqref{iteration} in the maximum norm we have the following result, which directly follows from  results presented in \cite{Gander2018}.
\begin{theorem} \label{thmconstraction}
The following holds:
\begin{equation} \label{estL1} \begin{split}
 \|L\|_\infty & = \|L_1\|_\infty  \leq C_1(\theta^\ast_1,\theta^\ast_2) <1  \\
 \text{with}~~& C_1(\theta^\ast_1,\theta^\ast_2):= \frac{\theta^\ast_2-\theta^\ast_1}{\pi}.
\end{split} \end{equation}
\end{theorem}
\begin{proof}
 Note $\|L\|_\infty= \max\{\, \|L_1\|_\infty,\|L_2\|_\infty\,\}$. We consider 
 \[ \|L_1\|_\infty=\max_{v \in C(\overline{\Gamma}_1)} \frac{\|L_1v\|_{L^\infty(\overline{\Gamma}_2)}}{\|v\|_{L^\infty(\overline{\Gamma}_1)}}.
 \]
Recall $L_1v=R_{\overline{\Gamma}_2}\Delta_1^{-1}(0,v)$. For $(\theta,r) \in B_1$ we have, with $\theta^\ast_1$ as in Fig.~\ref{FigGeom} and $\chi_{\Gamma_1}$ the characteristic function on $\Gamma_1$:
\begin{equation} \label{keystep} \begin{split}
  \big|\big(\Delta_1^{-1}(0,v)\big)(\theta,r)\big| & = \left| \frac{1}{2 \pi} \int_{- \theta^\ast_1}^{\theta^\ast_1} K(\theta-\theta',r) v(\theta')\, d\theta'\right| \\ & \leq  \frac{1}{2 \pi} \int_{- \theta^\ast_1}^{\theta^\ast_1} K(\theta-\theta',r) \, d\theta'\,  \|v\|_{L^\infty(\overline{\Gamma}_1)} \\
   &= ( \Delta_1^{-1}\chi_{\Gamma_1} )(\theta,r)\|v\|_{L^\infty(\overline{\Gamma}_1)}. 
\end{split} \end{equation}
Note that  $w = \Delta_1^{-1}\chi_{\Gamma_1} $ is the solution of the Laplace problem on $B_1$ with boundary data the piecewise constant function that has value $1$ on $\Gamma_1$ and $0$ on $\partial B_1 \setminus \Gamma_1$. This solution $w$ has constant values  on arcs of circles, cf. Example~\ref{exGander}.  Due to $R_{\overline{\Gamma}_2}$ we restrict to $(\theta,r) \in \Gamma_2$ which is one of these arcs. The value of $w$ on this arc is given by $C_1(\theta^\ast_1,\theta^\ast_2)$ as defined in \eqref{estL1}. This proves $\|L_1\|_\infty  \leq C_1(\theta^\ast_1,\theta^\ast_2)$. For $L_2$ we apply the same arguments, but using polar coordinates on the disc $B_2$. This leads to the solution of the Laplace problem on $B_2$ with boundary data the piecewise constant function that has value $1$ on $\Gamma_2$ and $0$ on $\partial B_2 \setminus \Gamma_2$. The constant value of the solution on the arc $\Gamma_1$ is determined by the angles $\pi - \theta^\ast_2$ (in $B_2$) and $\pi - \theta^\ast_1$ (in $B_1$) and given by $\frac{(\pi - \theta^\ast_1)-(\pi - \theta^\ast_2)}{\pi}=C_1(\theta^\ast_1,\theta^\ast_2)$. Hence, in this case, due to symmetry properties, we have $\|L_2\|_\infty=\|L_1\|_\infty$, which yields the equality result in \eqref{estL1}.
\end{proof}
\ \\[1ex]
We briefly comment on this elementary proof. There are two key ingredients, namely the kernel sign property \eqref{Max1}  that is used in the inequality in \eqref{keystep}, and the fact that the solution of a Laplace problem with boundary data 1 on $\Gamma_1$ and $0$ on $\partial B_1 \setminus \Gamma_1$ has a solution with  values on  $\Gamma_2$ that are bounded away from 1. These arguments apply in a much more general setting of elliptic partial differental equations on overlapping domains, as first elaborated in \cite{Lions1989}. 
\begin{remark} \label{RemAdditive}
   \rm If one considers a multiplicative Schwarz DD, then the convergence is determined by $L_1L_2$ (or $L_2L_1$). Note that $L^2=\begin{pmatrix}
   L_1L_2 & 0 \\ 0 & L_2L_1 \end{pmatrix}$ holds. From this we obtain  
   $\max \{\|L_1L_2\|_\infty,\|L_2 L_1\|_\infty\} = \|L^2\|_\infty \leq \|L\|_\infty^2 \leq C_1(\theta^\ast_1,\theta^\ast_2)^2$. Thus, as expected for this case with two subdomains, the multiplicative method has a  contraction number that is (not larger than) the square of the one of the additive method. 
  \end{remark}  
\section{Schwarz-Fourier domain decomposition method} \label{SectSchwarzF}
In this section we introduce an inexact version of the Schwarz method \eqref{DDschwarzit}. We approximate the boundary data on $\partial B_1$ in the real Fourier basis, using a finite sum as in \eqref{finitesum}. The corresponding space is given by
\begin{equation} \label{DefUN1}
  U_N^1:= \{\, \alpha_0 +\sum_{n=1}^N \alpha_n \cos (n \theta) + \beta_n \sin( n \theta)\,|\, \alpha_n, \beta_n \in \R\,\} \subset L^2(\partial B_1).
\end{equation}
We also use polar coordinates on $B_2$, i.e. a parametrization of the form $(\theta,r) \to (m,0)+r(\cos \theta, \sin \theta)$, $\theta \in [0,2\pi]$, $r \in [0,R]$. The analogue of $U_N^1$ on $\partial B_2$ is denoted by $U_N^2$. On the ball $B_2$ one may want to use a different $N$ value (depending on $R$). To simplify the presentation we use the same $N$ values on $B_1$ and $B_2$. The $L^2$-orthogonal projection $L^2(\partial \Omega_i) \to U_N^i$ is denoted by $P_N^i$, $i=1,2$. The \emph{inexact} version of  \eqref{DDschwarzit} that we consider is as follows. Given $u_i^0 \in C(\overline{B}_i)$, with ${u_i^0}_{|\partial B_i \setminus \Gamma_i}=g_i$, $i=1,2$, we determine for $n\geq 1$:\\
\begin{minipage}{0.5\textwidth}
 \begin{align*}
  \Delta u_1^n &=0\quad \text{in}~B_1\\
   u_1^n&=P_N^1(g_1,u_2^{n-1}) \quad \text{on}~\partial B_1 
 \end{align*}
\end{minipage}
\begin{minipage}{0.49\textwidth}
 \begin{equation} \label{DDschwarzinexact} \begin{split}
  \Delta u_2^n &=0\quad \text{in}~B_2\\
   u_2^n&=P_N^2(g_2,u_1^{n-1}) \quad \text{on}~\partial B_2 .
 \end{split} \end{equation}
\end{minipage}
\ \\[1ex]
We call this method the Schwarz-Fourier iteration. 
Once the Fourier projection $P_N^1(g_1,u_2^{n-1})(\theta)=\alpha_0 +\sum_{n=1}^N \alpha_n \cos (n \theta) + \beta_n \sin( n \theta)$ has been determined, the harmonic extension $u_1^n$ is directly available via 
\[
  u_1^n(\theta,r)=\alpha_0 +\sum_{n=1}^N r^n\big(\alpha_n \cos (n \theta) + \beta_n \sin( n \theta)\big),
\]
cf. \eqref{solution}. Clearly the same arguments apply for $P_N^2(g_2,u_1^{n-1})$ and $u_2^n$. \emph{In the rest of this paper we study convergence properties of the Schwarz-Fourier iteration}. Note that due to the use of the finite dimensional spaces $U_N^i$ this iteration \emph{discretizes} the given Laplace problem \eqref{Laplace2}. Hence, besides the (rate of) convergence of this iteration, there is another highly relevant issue, namely the accuracy of the discrete solution. Assume 
$\lim_{n \to \infty} u_1^n =u_1^\infty$, then for the (total) error in $u_1^n$ we have $u_{|B_1}- u_1^n=\big(u_{|B_1}- u_1^\infty\big)+\big(  u_1^\infty- u_1^n\big)=:e_{\rm discr}+e_{\rm iter}$. In this paper we only consider the iteration error $e_{\rm iter}$.
\begin{remark}\label{Remdiscr} \rm
 Concerning the discretization error in the Schwarz-Fourier method we note the following. Even for smooth boundary data $g$ the solution $u$ of \eqref{Laplace2} restricted to one of the subdomain boundaries $u_{|\partial B_i}$ is in general only continuous at the intersection points. Therefore one can not expect very fast (exponential) convergence, although the spectral Fourier method is used for approximating $u_{|\partial B_i}$. Here we do not analyze this discretization aspect further, but note that in the applications with the three-dimensional analogon of the Schwarz-Fourier method (using spherical harmonics) one typically uses (very) low degree spherical harmonics, cf. \cite{Stammetal,Quan2018}. 
\end{remark}
\ \\[1ex]
For the convergence analysis we formulate \eqref{DDschwarzinexact} analogous to \eqref{defL1}-\eqref{iteration}, but now with the Dirichlet to Dirichlet mapping 
\begin{equation} \label{defL1N} L_{1,N}v:=R_{\overline{\Gamma}_2}\Delta_1^{-1}P_N^1(0,v),
\end{equation}
 and similarly $L_{2,N}$. Hence, the Schwarz-Fourier method is given by:
\begin{equation} \label{iterationinexact}
 \begin{pmatrix}
  {u_1^n}_{|\overline{\Gamma}_2}\\
  {u_2^n}_{|\overline{\Gamma}_1}
 \end{pmatrix} = 
 \begin{pmatrix}
   0 & L_{1,N} \\ L_{2,N} & 0
  \end{pmatrix}
  \begin{pmatrix}
  {u_1^{n-1}}_{|\overline{\Gamma}_2}\\
  {u_2^{n-1}}_{|\overline{\Gamma}_1}
 \end{pmatrix}
   +  \begin{pmatrix}
       R_{\overline{\Gamma}_2}\Delta_1^{-1}P_N^1(g_1,0)\\
       R_{\overline{\Gamma}_1}\Delta_2^{-1}P_N^2(g_2,0)
      \end{pmatrix}.
\end{equation}
The operator $L_{N}:=\begin{pmatrix} 0 & L_{1,N} \\ L_{2,N} & 0
  \end{pmatrix}$ determines the convergence properties of the method and for deriving contraction results in the maximum norm one has to study (only) $\|L_{1,N}\|_\infty$. Note that $L_{2,N}$ as essentially the same structure as $L_{1,N}$. In the next section we first derive results for the operator $\Delta_1^{-1} P_N^1: L^2(\partial B_1) \to C^\infty(\overline{B}_1)$, which for a given boundary data function, first takes the finite dimensional Fourier projection of this function and then the corresponding harmonic extension on $B_1$ of this projection.    

\section{A variant of the maximum principle} \label{sectNewmaximum}
In this section we derive a variant of  the maximum principle for the operator $\Delta_1^{-1} P_N^1$. To avoid technical details we restrict the domain of this operator to the subspace $C_\cZ \subset L^2(\partial B_1)$, consisting of piecewise continuous functions on $\partial B_1$, cf. Section~\ref{sectpreliminaries}. Note that, for $g \in C_\cZ$:
\[
 w=\Delta_1^{-1} P_N^1 g \quad \text{iff}~ w(\theta,r)= \frac{1}{2 \pi} \int_0^{2 \pi} K(\theta',r) P_N^1 g(\theta-\theta')\, d\theta', \quad (\theta,r) \in B_1.  
\]
We can shift the projection operator to the Poisson kernel and then obtain the following result, cf. also \eqref{Poissonkernel}:
\begin{lemma} \label{shiftPN}
 For $g \in C_\cZ$  we have
 \begin{equation} \label{resultshift} \begin{split}
   w=\Delta_1^{-1} P_N^1 g \quad & \text{iff}~~ w(\theta,r)= \frac{1}{2 \pi} \int_0^{2 \pi} K_N(\theta',r)  g(\theta-\theta')\, d\theta', \quad (\theta,r) \in B_1, \\
   & \text{with}~~K_N(\psi,r)= 1+ 2 \sum_{n=1}^N \cos (n \psi) r^n.
 \end{split} \end{equation}
\end{lemma}
\begin{proof} 
Take $(\theta,r) \in B_1$. Then $r<1 $ and the series  $\sum_{n=1}^\infty \cos(n \psi) r^n$ converges uniformly in $\psi \in [0,2\pi]$. We have an explicit representation of the Fourier series of $\psi \to K(\psi,r)$ as in \eqref{Poissonkernel} and  the $L^2$-orthogonal projection
on the space $U_N^1$, cf. \eqref{DefUN1}, is given by $K_N(\cdot,r):=P_N^1 K( \cdot,r)= 1+ 2 \sum_{n=1}^N \cos (n\, \cdot) r^n$. Hence,
\begin{align*} 
  \int_0^{2 \pi} K(\theta',r) P_N^1 g(\theta-\theta')\, d\theta' & =\int_0^{2 \pi} P_N^1 K(\theta',r)  g(\theta-\theta')\, d\theta' \\ & =
  \int_0^{2 \pi} K_N(\theta',r)  g(\theta-\theta')\, d\theta',
\end{align*}
which completes the proof. 
\end{proof}
\ \\
Recall that a necessary condition for the maximum principle \eqref{maximum} to hold, is the sign property $K(\psi,r) \geq 0$ for all $(\psi,r) \in B_1$. Due to the oscillating behavior of a Fourier approximation near a discontinuity (Gibbs phenomenon) we do not expect $K_N$ to have such a sign property. 
For the values $N=5$, $N=25$, we illustrate $K_N$ in Figure~\ref{FigKN}. 
\begin{figure}[ht!]
	\centering\includegraphics[width=0.45\textwidth]{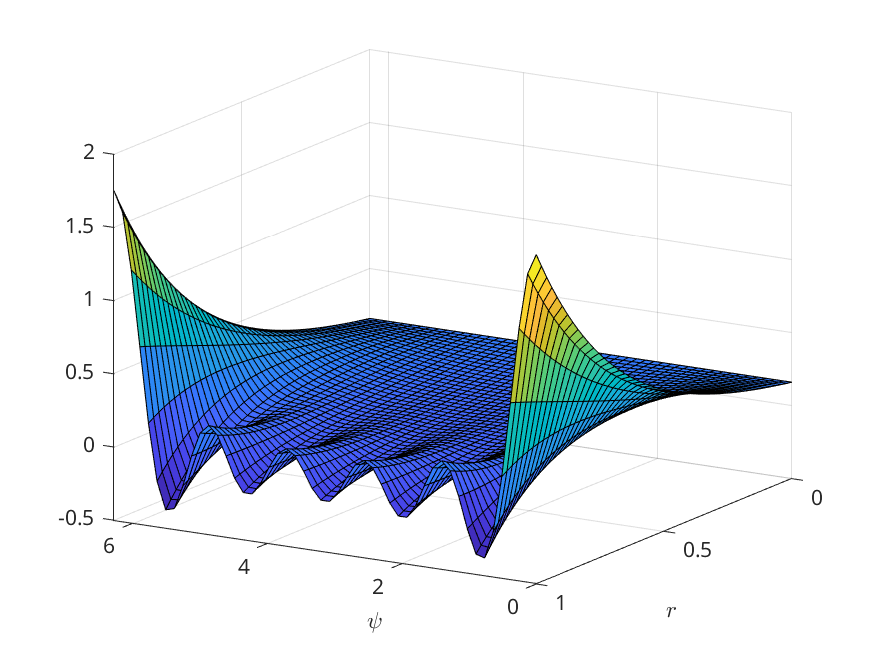}\qquad 
	\includegraphics[width=0.45\textwidth]{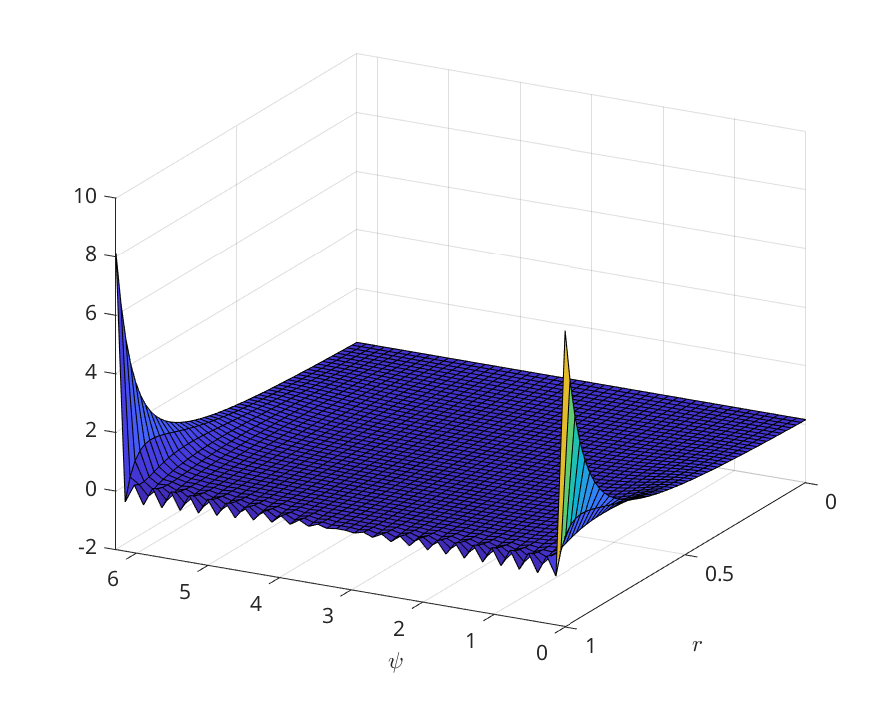}
	\caption{Projected kernels $K_N(\psi,r)$ for $N=5$ (left) and $N=25$ (right). \label{FigKN}}
\end{figure}

In the theorem below we show that $K_N(\theta,r)$ \emph{is positive} at all points $(\theta,r) \in B_1$ that have a distance at least $\sim \frac{\ln N}{N}$ to the boundary $\partial B_1$.
\begin{theorem} \label{mainthm} The following holds for $N\geq 4$:
\begin{equation} \label{mainestimate} \begin{split}
  K_N(\theta,r) & \geq 0 \quad \text{for all}~~(\theta,r) \in B_1 \quad \text{with} ~~ r \leq r_N^\ast, \\
  \text{with}~~r_N^\ast & :=\left(1- 2 \frac{\ln\big(2(N+1)\big)}{N+1} \right)^\frac12.
\end{split} \end{equation}
\end{theorem}
\begin{proof}
Take $\theta \in [0,2 \pi]$, $r \in [0,1)$. With $z:=r e^{i\theta}$ we have $\bar z z=r^2$ and we obtain
\begin{align*}
 K_N(\theta,r)& = 1 + 2 \sum_{n=1}^N \cos (n \theta) r^n=-1 + 2 \sum_{n=0}^N \cos (n \theta) r^n \\
 &= -1+ \sum_{n=0}^N \big(z^n +\bar z^n\big) = -1 +\frac{1-z^{N+1}}{1-z} +\frac{1-\bar z^{N+1}}{1-\bar z}\\
 &= \frac{1-r^2-\big(z^{N+1}+\bar z^{N+1}\big) +r^2\big(z^N +\bar z^N\big)}{(1-z)(1-\bar z)}.
\end{align*}
For the denominator we have
\[
 (1-z)(1-\bar z)= 1-2 r \cos(\theta) + r^2 \geq 1-2r+r^2=(1-r)^2 >0 \quad \text{for all}~r \in [0,1).
\]
It remains to analyze the sign of the nominator. For this we note
\begin{align*}
  & 1-r^2-\big(z^{N+1}+\bar z^{N+1}\big) +r^2\big(z^N +\bar z^N\big) \\ & = 1-r^2-2 r^{N+1} \cos\big((N+1)\theta\big)+ 
  2 r^{N+2}\cos(N\theta)\\
  & \geq 1-r^2 -2 r^{N+1} -2 r^{N+2} \geq 1-r^2 -4 r^{N+1} \geq 1-r^2-4(r^2)^x,
\end{align*}
with $x:=\tfrac12 (N+1)\geq \tfrac52$. We substitute $y=1-r^2$ and study for which $y=y(x)$, with $x \geq \tfrac52$, the inequality 
\begin{equation} \label{hulpeq} y-4(1-y)^x \geq 0
\end{equation}
 holds. For this it is convenient to use the Lambert function. For given $z \geq 0$, $W(z)$ is the unique solution of $W(z)e^{W(z)}=z$.
 We define $y_0:=4 e^{-W(4x)}$. Hence, $W(4x)=\ln \big(\frac{4}{y_0}\big)$ and $W(4x) e^{W(4x)}=4x$ hold. This yields $\ln\big(\frac{4}{y_0}\big) \frac{4}{y_0}=4x$, which can be rewritten as
 \begin{equation} \label{hulp9}
    \ln\big(\frac{y_0}{4}\big) + x y_0=0.  
 \end{equation}
For $z \geq e^{-1}$ the property $W(z \ln z)=\ln z$ holds. With $z=4$ and monotonicity of the $W$-function we get $W(4x) \geq W(10)> W(5 \ln 5) = \ln 5
$. Hence $ 0 < y_0=4 e^{-W(4x)} < 4 e^{-\ln 5} < 1$ holds. Thus the inequality $\ln (1-y_0) \leq -y_0$ holds. Combining this with \eqref{hulp9}
we obtain 
\[
  0 =  \ln\big(\frac{y_0}{4}\big) + x y_0\leq \ln\big(\frac{y_0}{4}\big) - x \ln (1-y_0).
\]
This implies $\ln (1-y_0)^x \leq \ln\big(\frac{y_0}{4}\big)$ and thus
\begin{equation} \label{defy}
  y_0-4(1-y_0)^x \geq 0
\end{equation}
holds for $y_0=y_0(x)=4 e^{-W(4x)}$. Now note that $z \to z-4(1-z)^x$ is monotonically increasing on $[0,1)$. Hence we have for all $\hat y_0$ with $y_0 \leq \hat y_0 <1$:
\begin{equation} \label{H8}
  y-4(1-y)^x \geq 0 \quad \text{for all}~~y \in [\hat y_0,1).
\end{equation}
We now use known growth relations of the Lambert function to estimate $y_0$. The following (sharp) estimates are from \cite[Theorem 2.1]{Hoorfar2008}:
\[
   \ln z - \ln \ln z \leq W(z) \leq \ln z -\tfrac12  \ln \ln z \quad \text{for}~~z \geq e.
\]
Using the lower bound we obtain $y_0= 4 e^{-W(4x)}\leq 4 e^{- \ln (4x) +\ln \ln (4x)}=\frac{\ln (4x)}{x}=: \hat y_0$. To guarantee $\hat y_0 <1$ we need $N \geq 4$. We use this result in \eqref{H8}, use $x=\tfrac12 (N+1)$ and substitute $y=1-r^2$. This finally yields
\[
  1-r^2-4(r^2)^{\tfrac12(N+1)} \geq 0 \quad \text{if}~~r^2 \leq 1- 2 \frac{\ln\big(2(N+1)\big)}{N+1},
\]
which completes the proof. 
\end{proof}
\ \\[1ex]
Note that, since $(1-\epsilon)^\frac12 = 1-\tfrac12 \epsilon +\cO(\epsilon^2)$ ($\epsilon \to 0$),  for not too small $N$ values the bound in \eqref{mainestimate} is of the form $1- \frac{\ln (2(N+1))}{N+1}$. 
We study the sharpness of the bound \eqref{mainestimate}.  For this we numerically determine $\min \{\, \delta\,|\, K_N(\theta,r) \geq 0~~\text{for all}~\theta \in [0,2 \pi],~r \leq 1-\delta\,\}$. This numerical value is denoted by $\delta_{\rm num}(N)$, i.e., $K_N(\theta,r) \geq 0 $ for all $\theta \in [0,2 \pi]$ and $r \leq 1-\delta_{\rm num}(N)$ (in the numerical computation). The result in Theorem~\ref{mainthm} proves $K_N(\theta,r) \geq 0 $ for all $\theta \in [0,2 \pi]$ and $r \leq 1- \delta_{\rm th}(N)$ with $\delta_{\rm th}(N):=1-r_N^\ast$. Clearly $ \delta_{\rm th}(N) \geq  \delta_{\rm num}(N)$ must hold.  It is convenient to compare the  inverses, i.e.,
$1/\delta_{\rm th}(N)$ and $1/\delta_{\rm num}(N)$. For discrete $N \in \nn$ values the corresponding values of $1/\delta_{\rm num}(N)$ are shown in Fig.~\ref{Figcompare} (left). In the same figure we show the function $N \to 1/\delta_{\rm th}(N)$, $N \in [5,100]$. In the right figure we compare the two bounds using $q(N):= \frac{\delta_{\rm num}(N)}{\delta_{\rm th}(N)}$.

\begin{figure}[ht!]
	\centering\includegraphics[width=0.45\textwidth]{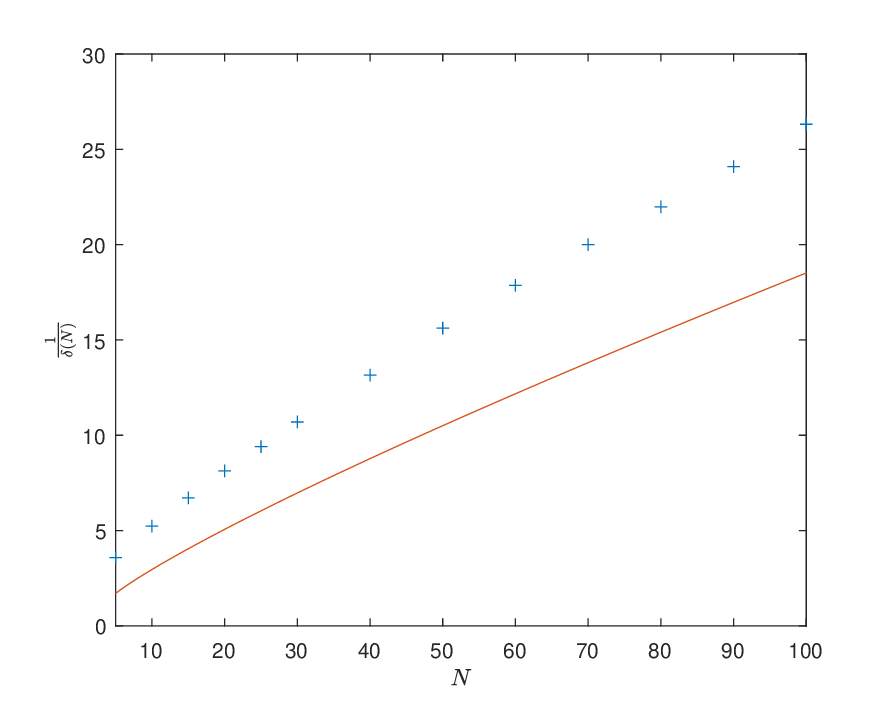}\qquad 
	\includegraphics[width=0.45\textwidth]{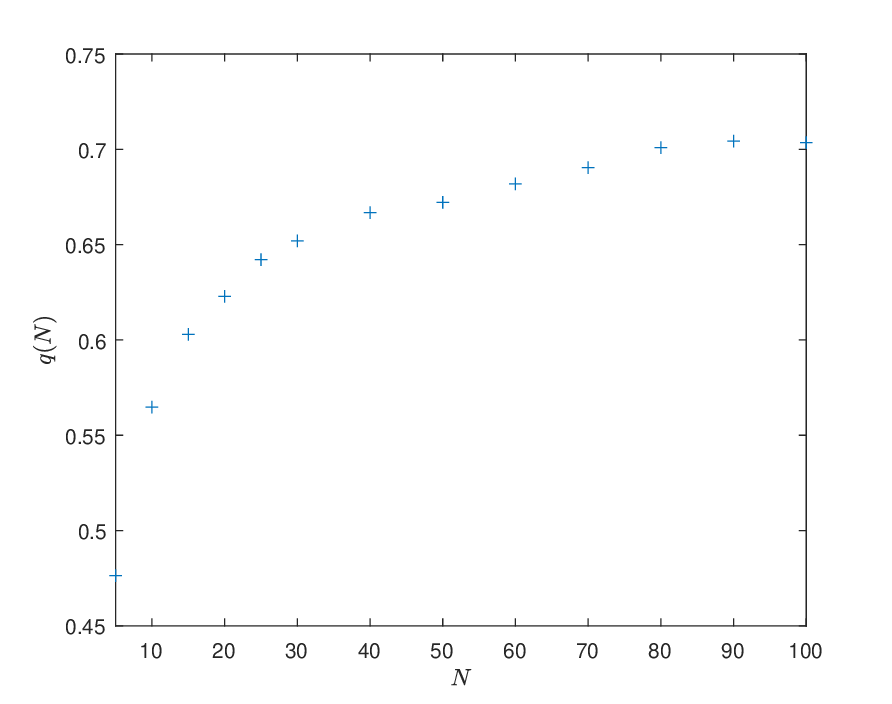}
	\caption{$1/\delta_{\rm num}(N)$ ($+$) and  $1/\delta_{\rm th}(N)$ (curve); $q(N)=\frac{\delta_{\rm num}(N)}{\delta_{\rm th}(N)}$. \label{Figcompare}}
\end{figure}

From these results we conclude that the bound derived in \eqref{mainestimate} is (very) satisfactory. Based on the results for the quotient $q(N)$ we expect that the logarithmic growth factor in \eqref{mainestimate} can not be avoided. 

From the result in Theorem~\ref{mainthm} we immediately obtain a variant of the maximum principle for the case that boundary data are projected into the finite dimensional Fourier space $U_N^1$. For this we introduce the notation
\[
  B_1^+(N):= \left\{\, (\theta,r)\in B_1\,|\, \theta \in [0,2 \pi], ~0 \leq r \leq r_N^\ast\,\right\}, \quad N \geq 4. 
\]
for the subdomain of $B_1$ where the projected kernel $K_N$ is positive (expressed by the superscript $+$ in $B_1^+(N)$). Combining the above results we obtain the following main theorem, cf. the maximum principle \eqref{maximum}--\eqref{Max2}. 
\begin{theorem} \label{CorComparison} 
The following holds for $N\geq 4$:
\begin{align} 
  K_N(\theta,r) & \geq 0 \quad \text{for all}~~(\theta,r) \in B_1^+(N), \label{main1}\\
  \frac{1}{2 \pi} \int_0^{2 \pi}  K_N(\psi,r) \, d \psi & = 1 \quad \text{for all}~ 0 \leq r < 1, \label{main2}
\\
 \forall ~g~ \in C_\cZ:~ &  \min_{\partial B_1} g \leq  (\Delta_1^{-1}P_N^1 g)(\theta,r) \leq  \max_{\partial B_1}g~~ \forall~(\theta,r) \in B_1^+(N).  \label{maxvariant}
  \end{align}
\end{theorem}
\begin{proof}
The result \eqref{main1} is given in Theorem~\ref{mainthm}.
Note that $K_N=P_N^1 K$ and thus
\[
  \frac{1}{2 \pi} \int_0^{2 \pi}  K_N(\theta',r) \, d \theta'= \frac{1}{2 \pi} \int_0^{2 \pi}  K(\theta',r) P_N^1 1 \, d \theta' = \frac{1}{2 \pi} \int_0^{2 \pi}  K(\theta',r) \, d \theta'=1,
\]
which yields \eqref{main2}. Note that, cf. Lemma~\ref{shiftPN},
\[
  (\Delta_1^{-1} P_N^1 g)(\theta,r)= \frac{1}{2 \pi} \int_0^{2 \pi} K_N(\theta',r)  g(\theta-\theta')\, d\theta', \quad (\theta,r) \in B_1,
\]
holds. For $(\theta,r) \in B_1^+(N)$ we have $K_N(\theta',r) \geq 0$ for all $\theta' \in [0,2\pi]$. This and \eqref{main2} imply the result \eqref{maxvariant}. 
\end{proof}
\begin{remark} \rm As far as we know the variant of the maximum principle given  in Theorem~\ref{CorComparison} is new. The result in \eqref{maxvariant} shows that for the inverse Laplacian combined with the Fourier projection, $\Delta_1^{-1} P_N^1$, we have a maximum principle, provided we restrict to the subdomain $B_1^+(N)$ of $B_1$. An estimate for the whole domain $B_1$, using the maximum principle for the Laplacian, is
\begin{equation} \label{estwhole}
   \|\Delta_1^{-1} P_N^1 g\|_{L^\infty(B_1)} \leq \|P_N^1 g\|_{L^\infty(\partial B_1)} \leq \big( \frac{4}{\pi^2} \ln N + \cO(1)\big)\|g\|_{L^\infty(\partial B_1)},
\end{equation}
where the latter bound follows from $(P_N^1 g)(\theta) = \frac{1}{2\pi} \int_{-\pi}^{\pi} D_N(\theta')g(\theta-\theta') \,d \theta'$ and the
property $\frac{1}{2 \pi}\int_{-\pi}^\pi |D_N(\theta')| \, d\theta'=  \frac{4}{\pi^2} \ln N + \cO(1) $ of  the  Dirichlet kernel $D_N$. Note that both inequalities in \eqref{estwhole} are sharp and that the resulting factor $\frac{4}{\pi^2} \ln N + \cO(1)$ in the upper bound is not only larger than 1 but blows up for $N \to \infty$. In  \eqref{maxvariant} we have a constant 1 in the bound, but have to restrict to a subdomain $B_1^+(N) \subset B_1$ with ${\rm dist} (B_1^+(N), \partial B_1) \sim \frac{\ln N}{N}$.  
\end{remark}
 \ \\[1ex]
 We can use this maximum principle result to analyze the convergence of the  Schwarz-Fourier domain decomposition method along the same lines as in Theorem~\ref{thmconstraction}. The contraction number (in the maximum norm) of this method is determined by
\begin{equation} \label{esti} \|L_{1,N}\|_\infty=\max_{v \in C(\overline{\Gamma}_1)} \frac{\|L_{1,N}v\|_{L^\infty(\overline{\Gamma}_2)}}{\|v\|_{L^\infty(\overline{\Gamma}_1)}}=\max_{v \in C(\overline{\Gamma}_1)} \frac{\|R_{\overline{\Gamma}_2} \Delta_1^{-1} P_N^1(0,v)\|_{L^\infty(\overline{\Gamma}_2)}}{\|v\|_{L^\infty(\overline{\Gamma}_1)}},
 \end{equation}
 cf. \eqref{iterationinexact}. To be able to use the estimate  \eqref{maxvariant} we have to restrict to $B_1^+(N)$. Therefore, 
 we split $\overline{\Gamma}_2 \subset B_1$ into $\Gamma_2^{(N)}:=\Gamma_2 \cap B_1^+(N)$ and $\overline{\Gamma}_2 \setminus \Gamma_2^{(N)}$
 and use
 \begin{equation}  \label{K7} \begin{split}
 &  \|R_{\overline{\Gamma}_2} \Delta_1^{-1} P_N^1(0,v)\|_{L^\infty(\overline{\Gamma}_2)} \\ & = \max \left\{ \|R_{\overline{\Gamma}_2} \Delta_1^{-1} P_N^1(0,v)\|_{L^\infty({\Gamma}_2^{(N)})}   ,\|R_{\overline{\Gamma}_2} \Delta_1^{-1} P_N^1(0,v)\|_{L^\infty(\overline{\Gamma}_2\setminus {\Gamma}_2^{(N)})} \right\} 
 \end{split} \end{equation} 
Note that ${\rm meas}(\overline{\Gamma}_2 \setminus \Gamma_2^{(N)}) \sim \frac{\ln N}{N} \to 0$ for $N \to \infty$. The term $\|R_{\overline{\Gamma}_2} \Delta_1^{-1} P_N^1(0,v)\|_{L^\infty(\Gamma_2^{(N)})}$, cf. \eqref{esti}  can be bounded using the maximum principle.
\begin{theorem} \label{lemmaxreduced}
  The following holds for $N \geq 4$, with $C_1(\theta^\ast_1,\theta^\ast_2)$ as in \eqref{estL1} and $r^\ast_N$ as in \eqref{mainestimate}:
  \begin{equation}\label{estM} \begin{split}
   \max_{v \in C(\overline{\Gamma}_1)} \frac{\|R_{\overline{\Gamma}_2} \Delta_1^{-1} P_N^1(0,v)\|_{L^\infty(\Gamma_2^{(N)})}}{\|v\|_{L^\infty(\overline{\Gamma}_1)}} &  \leq C_1(\theta^\ast_1,\theta^\ast_2) + \epsilon(r_N^\ast),   \\
   \text{with}~~ \epsilon(r_N^\ast):= \frac{2}{\pi} \int_0^{r_N^\ast} \frac{s^N}{1-s} \, ds.
  \end{split} \end{equation}
\end{theorem}
\begin{proof}
Take $v \in C(\overline{\Gamma}_1)$. This function extended by 0 on $\partial B_1$ is denoted by $v^{\rm ex}$, i.e., $v^{\rm ex}=(0,v)$. Due to Lemma~\ref{shiftPN} and Theorem~\ref{CorComparison} we can apply the same arguments as in \eqref{keystep}, which yields, for $(\theta,r) \in \Gamma_2^{(N)}$:
\begin{equation} \label{H9} \begin{split}
  |(\Delta_1^{-1} P_N^1 (0,v))(\theta,r)|& = \left| \frac{1}{2 \pi} \int_{-\theta^\ast_1}^{\theta^\ast_1} K_N(\theta-\theta',r)  v^{\rm ex}(\theta')\, d\theta'\right| \\
  &  \leq  \frac{1}{2 \pi} \int_{-\theta^\ast_1}^{\theta^\ast_1} K_N(\theta-\theta',r) \, d\theta'\, \|v\|_{L^\infty(\overline{\Gamma}_1)} \\
   & =\frac{1}{2 \pi} \int_{0}^{2 \pi} K(\theta-\theta',r) (P_N^1 \chi_{\Gamma_1})(\theta')\, d\theta' \, \|v\|_{L^\infty(\overline{\Gamma}_1)} \\ 
    & = (\Delta_1^{-1} P_N^1 \chi_{\Gamma_1})(\theta,r)\, \|v\|_{L^\infty(\overline{\Gamma}_1)} .
\end{split}
\end{equation}
The function $w_N =\Delta_1^{-1} P_N^1 \chi_{\Gamma_1}$ is the solution of the Laplace equation on $B_1$ with boundary data $P_N^1 \chi_{\Gamma_1}$. A straightforward computation yields $P_N^1 \chi_{\Gamma_1}(\theta')=\frac{1}{\pi}\big(\theta^\ast_1 + \sum_{n=1}^N \frac{2}{n} \sin(n \theta^\ast_1) \cos (n \theta')\big)$. Hence we have the representation, for $(\theta,r) \in B_1$: 
\begin{equation} \label{H2} \begin{split} 
w_N(\theta,r)  & = \frac{1}{\pi}\big(\theta^\ast_1 + \sum_{n=1}^N \frac{2}{n}\sin(n \theta^\ast_1) \cos (n \theta) r^n\big) \\
 & = w_\infty(\theta,r)  -\frac{1}{\pi} \sum_{n=N+1}^\infty \frac{2}{n}\sin(n \theta^\ast_1) \cos (n \theta) r^n   =: w_\infty(\theta,r) -e_N(\theta,r).
\end{split} \end{equation}
For the limit solution $w_\infty(\theta,r)$ we have, cf. proof of Theorem~\ref{thmconstraction}, $|w_\infty(\theta,r)| \leq C_1(\theta^\ast_1,\theta^\ast_2)$ for all $(\theta,r) \in \Gamma_2$. Take $(\theta,r) \in \Gamma_2^{(N)}$, hence $r \leq r_N^\ast$ and thus
\[
  |e_N(\theta,r)| \leq \frac{1}{\pi}\sum_{n=N+1}^\infty \frac{2}{n} (r_N^\ast)^n=:g_N(r_N^\ast),
\]
with $g_N(x):=\frac{1}{\pi}\sum_{n=N+1}^\infty \frac{2}{n} x^n$, $x \in [0,1)$. Using $g_N'(x)=\frac{2}{\pi} \frac{x^N}{1-x}$ we get 
$g_N(r_N^\ast)= \frac{2}{\pi} \int_0^{r_N^\ast} \frac{s^N}{1-s} \, ds$. Combining this with the results above completes the proof.  
\end{proof}
\ \\[1ex]
We discuss the term $\epsilon(r_N^\ast):= \frac{2}{\pi} \int_0^{r_N^\ast} \frac{s^N}{1-s} \, ds$ that occurs in \eqref{estM}.
Numerical computation with a sufficiently accurate quadrature rule yield the results shown in Table~\ref{Tabeps}.
\begin{table}[ht!]
\begin{tabular}{ | l | l | l | l | l | l | l| l|}
\hline 
  $N$  & 5 & 10 & 20 & 30 & 40 & 60  & 80 \\ 
  \hline 
 $\epsilon(r_N^\ast)$  & $0.00084$ & $0.0016$ & $0.0013$ & $0.0010$ & $0.00082$ & $0.00059$  & $0.00046$ \\
  \hline
 bound \eqref{est5} &  $0.11$ & $0.056$ & $0.033$ & $0.024$ & $0.020$ & $0.014$  & $0.012$ \\
  \hline
\end{tabular}
\caption{Values $\epsilon(r_N^\ast)$ computed with quadrature. \label{Tabeps}}
\end{table}
We see that  $\epsilon(r_N^\ast) \ll 1$ holds, also  for small $N$,  and that $\epsilon(r_N^\ast)$ decreases for $N \geq 10$. The latter is indeed the case, as is shown in the following estimate.
\begin{lemma} For $N \geq 4$ we have with $\alpha(N):=2 \frac{\ln (2(N+1))}{N+1}$
 \begin{equation} \label{est5}
  \epsilon(r_N^\ast) \leq \frac{1}{\pi}\frac{\ln \big(2 \alpha(N)^{-1}\big)}{(1-\alpha(N))^\frac12} \frac{1}{N+1}. 
 \end{equation}
\end{lemma}
\begin{proof}
From partial integration we obtain
\begin{equation} \label{est55}
  \epsilon(r_N^\ast)= \frac{2}{\pi} \int_0^{r_N^\ast} \frac{s^N}{1-s} \, ds \leq \frac{2}{\pi} \ln \left(\frac{1}{1-r_N^\ast}\right) (r_N^\ast)^N.
\end{equation}
Using $\sqrt{1-\delta} \leq 1-\tfrac12 \delta$ for $\delta \in [0,1]$, we obtain
\[
  \ln \left(\frac{1}{1-r_N^\ast}\right) =  \ln \left(\frac{1}{1-\big(1- \alpha(N)\big)^\frac12}\right) \leq \ln \big(2\alpha(N)^{-1}\big).
\]
Furthermore, using $(1-\frac{1}{y})^y \leq \frac{1}{e}$ for $y \geq 1$ we obtain
\[ \begin{split}
 (r_N^\ast)^N & = (1- \alpha(N))^{\frac12 N}=(1- \alpha(N))^{-\frac12}(1- \alpha(N))^{\frac12 (N+1)} \\
  & =(1- \alpha(N))^{-\frac12} \big((1- \alpha(N))^{1/\alpha(N)}\big)^{\ln(2(N+1))} \\
  & \leq (1- \alpha(N))^{-\frac12}   e^{-\ln(2(N+1))}= (1- \alpha(N))^{-\frac12} \frac{1}{2(N+1)}.
\end{split} \]
Combining these estimates completes the proof.
\end{proof}
\ \\[1ex]
Values for the upper bound in \eqref{est5} are shown in Table~\ref{Tabeps}.
We note that for these results to hold it is essential that we have the $\ln$-term in $r_N^\ast$. If we repeat the computation of Table~\ref{Tabeps} with $r_N^\ast$ replaced by $\tilde r_N^\ast:=\left(1-\frac{2}{N+1}\right)^\frac12$ we obtain values for $\epsilon(\tilde r_N^\ast)$ that are $0.12$ for $N=5$, $0.13$ for $N\in\{10,20\}$ and $0.14$ for $N \in \{30,40,60,80\}$. 

We comment on the result in Theorem~\ref{lemmaxreduced}. The results in the first row of Table~\ref{Tabeps} show that we can neglect the term $\epsilon(r_N^\ast)$ in \eqref{estM}. Thus the result shows that, if for the Schwarz-Fourier iteration  one restricts the range space of the Dirichlet to Dirichlet mapping $L_{1,N}$ to $\Gamma_2^{(N)} \subset \Gamma_2$, we have (essentially) the same maximum norm bound $C_1(\theta_1^\ast,\theta_2^\ast)$ as for the continuous case, cf. \eqref{estL1}.  
We give a heuristic explanation of this similarity of bounds. Due the key property $K_N \geq 0$ on $\Gamma_2^{(N)}$ we can use the same maximum principle arguments in \eqref{H9} as in \eqref{keystep}. In the former we then obtain $\Delta_1^{-1} P_N^1\chi_{\Gamma_1}$ instead of $\Delta_1^{-1} \chi_{\Gamma_1}$. The difference, $(I-P_N)\chi_{\Gamma_1}$, consists of higher ($\geq N+1$) frequencies that are strongly damped in the inverse Laplacian $\Delta_1^{-1}$ and become very small at distance $1-r_N^\ast$ from the boundary $\partial B_1$. The latter corresponds to the very small values for $\epsilon(r_N^\ast)$. 
\begin{example} \label{ex1} \rm
In a numerical experiment we evaluate the Dirchlet to Dirichlet map $L_{1,N}v=R_{\overline{\Gamma}_2} \Delta_1^{-1}P_N^1(0,v)$ for $v \equiv 1$, i.e., $R_{\overline{\Gamma}_2} \Delta_1^{-1}P_N^1 \chi_{\Gamma_1}$. Note that this function also occurs in the key estimate \eqref{H9}. 
First for $B_2$ we take the disc with center $m=1.4$ and radius $R=1.2$. In this case we have a large overlap: the maximal ball contained  in the overlap region $B_1 \cap B_2$ has diameter $0.8$. For this case we have $C_1(\theta^\ast_1,\theta^\ast_2)=0.436$. Results are shown in Fig.~\ref{FigExpPN} (left). On the $x$-axis the values of $\tilde \theta$ that parameterize $\Gamma_2$, i.e., $\tilde \theta \in [\theta^\ast_2, 2\pi -\theta^\ast_2]$, are given. On the $y$-axis we have the values $(\Delta_1^{-1}P_N^1 \chi_{\Gamma_1})(\tilde \theta)$. For this case, already for small $N$ values a large part of $\Gamma_2$ is in $B_1^+(N)$ and we observe that in that region, already for small $N$ values, the computed values are (very) close to  $C_1(\theta^\ast_1,\theta^\ast_2)$. This shows that the result \eqref{estM} is sharp. In the right figure we show results for $m=2.1$ and radius $R=1.2$, where the diameter of the maximal ball contained in $B_1 \cap B_2$ has a relatively small value $0.1$. For this case we have $C_1(\theta^\ast_1,\theta^\ast_2)=0.807$. We observe a similar behavior, but larger $N$ values are needed to have values close to $C_1(\theta^\ast_1,\theta^\ast_2)$ in the region $\Gamma_2 \cap B_1^+(N)$. In both cases we observe that close to the intersection points $z_1$, $z_2$ (endpoints of the interval $[ \theta^\ast_2, 2\pi - \theta^\ast_2]$) we have values close to $0.5$. This is expected, since $\lim_{N \to \infty}(P_N^1 \chi_{\Gamma_1})(z_i)=0.5$, $i=1,2$, cf.~\eqref{prop1}. We also refer to Remark~\ref{RemEndpoint}.  
\begin{figure}[ht!]
	\centering\includegraphics[width=0.45\textwidth]{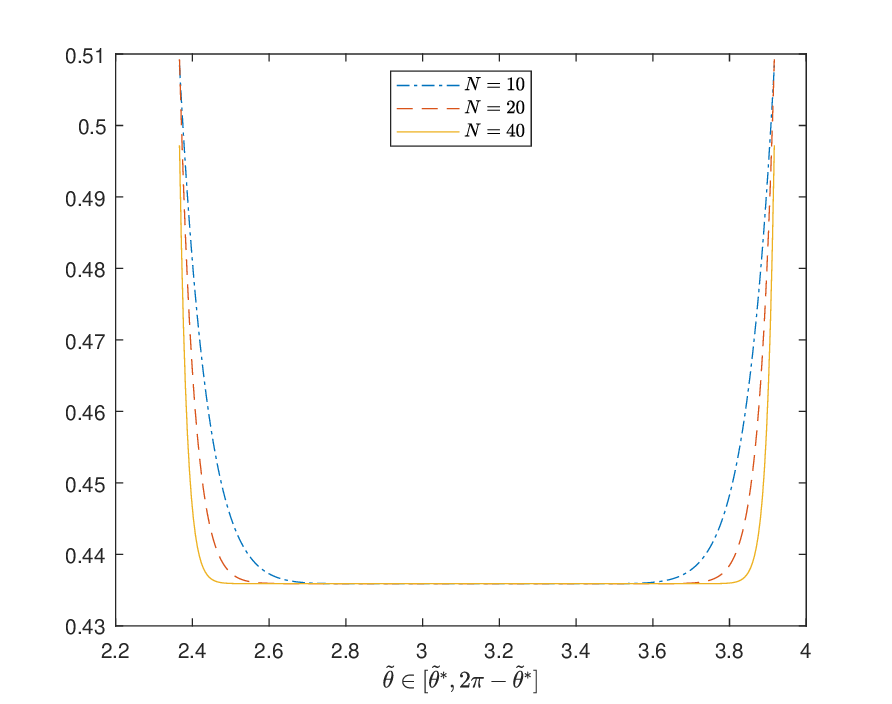}\qquad 
	\includegraphics[width=0.45\textwidth]{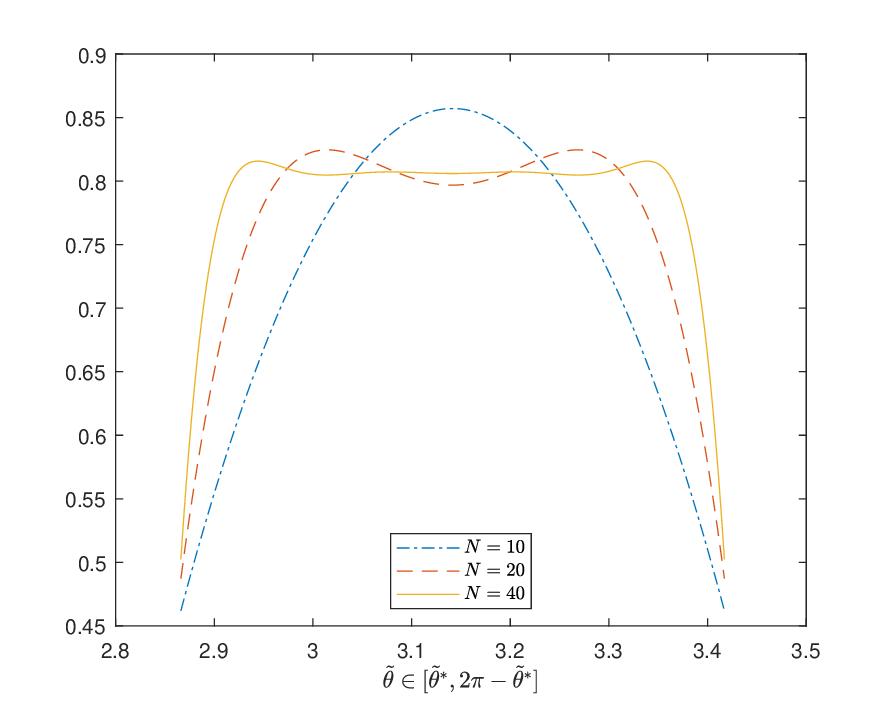}
	\caption{$R_{\overline{\Gamma}_2} \Delta_1^{-1}P_N^1 \chi_{\Gamma_1}$, for $m=1.4$, $R=1.2$ (left) and $m=2.1$, $R=1.2$ (right). \label{FigExpPN}}
\end{figure}
\end{example}

For deriving rigorous maximum norm bounds for the Dirichlet to Dirichlet map $L_{1,N}$ it remains to derive bounds on the boundary layer regions, i.e., bounds for  $\|R_{\overline{\Gamma}_2} \Delta_1^{-1} P_N^1(0,v)\|_{L^\infty(\overline{\Gamma}_2\setminus {\Gamma}_2^{(N)})}$, cf. \eqref{K7}. Unfortunately, we are not able to derive satisfactory  bounds. This is one of the reasons, why in the next section we study a variant of the Schwarz-Fourier method \eqref{DDschwarzinexact} in which we replace the projections $P_N^i$ by a nodal interpolation, which corresponds to using a discrete Fourier transform. Further reasons for studying this variant are that the discrete Fourier transform allows a very efficient implementation and that we can avoid discontinuities at the intersection points $z_i$. The final, most important, reason is that for this variant we can determine explicit bounds for the maximum norm of the corresponding Dirichlet to Dirichlet map  that can be computed numerically.  
\section{Schwarz-Fourier interpolation} \label{sectdiscrete}
In this section we study a variant of te Schwarz-Fourier iteration \eqref{DDschwarzinexact} in which the projections $P_N^i$, $i=1,2$, are replaced by the Fourier interpolation operators. This version can be very efficiently implemented using FFT and allows a different convergence analysis. For the latter, however, we have to assume that the intersection points $z_i$ coincide with interpolation points. Due to this we have to restrict to a subset of all possible intersection scenarios, parameterized by the angles $\theta^\ast_1, \theta^\ast_2$. We first introduce some notation.

As above we use polar coordinates to parameterize $\partial B_i$ and use interpolation on a uniform grid along the arclength variable. We use $n_i$ equidistant interpolation points on $\partial B_i$, $i=1,2$. To simplify the presentation we assume $n_i$ to be even. We introduce the grids
\begin{align*} 
  G_{1,n_1} & =\big\{ X_{\ell}\,|\, X_\ell=\big( \cos(\ell \tfrac{2 \pi}{n_1}), \sin(\ell \tfrac{2 \pi}{n_1})\big)\quad  \ell=0,1, \ldots, n_1-1\,\big\}, \\
  G_{2,n_2} & =\big\{ X_{\ell}\,|\, X_\ell=(m,0)+ R\big( \cos(\ell \tfrac{2 \pi}{n_2}), \sin(\ell \tfrac{2 \pi}{n_2})\big), \quad \ell=0,1, \ldots, n_2-1\,\big\}.
\end{align*}
In the remainder we assume that the following holds. 
\begin{assumption} \label{assGrid} \rm  We  assume that for  the intersection point $z_1=(\cos \theta_1^\ast, \sin \theta_1^\ast)= (m,0)+R(\cos  \theta_2^\ast, \sin \theta_2^\ast)$ we have
\begin{equation} \label{cond}
    z_1 \in G_{1,n_1}\cap G_{2,n_2}.
\end{equation}
\end{assumption}

Note that since the interpolation points on the circles $\partial B_i$ are symmetric with respect to the $x$-axis,  the other intersection point $z_2$ then also has the property $z_2 \in G_{1,n_1}\cap G_{2,n_2}$.
\begin{remark} \label{remGrid}
 \rm We comment on Assumption~\ref{assGrid}. Let $n_1$ and $n_2$ be given. Recall that the set of all intersection scenarios is described by  angles $\theta_1^\ast$, $\theta_2^\ast$, with $0 < \theta_1^\ast \leq \theta_2^\ast <\pi$, cf. Section~\ref{sectDDSchwarz}. Assumption~\ref{assGrid} interpolates the set of all scenarios  in a certain sense. 
For the intersection point $z_1=(\cos \theta_1^\ast, \sin \theta_1^\ast)= (m,0)+R(\cos  \theta_2^\ast, \sin \theta_2^\ast)$ the condition \eqref{cond}  is satisfied iff $\theta_1^\ast= \ell \tfrac{2 \pi}{n_1}$ with $1 \leq \ell < \tfrac12 n_1$  and $ \theta_2^\ast= \tilde \ell \tfrac{2 \pi}{n_2}$ with $\ell \tfrac{n_2}{n_1} \leq  \tilde \ell < \tfrac12 n_2$. Hence, the collection of all possible intersection scenarios is reduced to this finite number of scenarios. Given an arbitrary scenario parameterized by angle values $\theta_1^\ast$ and $\theta_2^\ast$, a corresponding approximate scenario that satisfies Assumption~\ref{assGrid} is obtained by choosing  first a closest grid point to $\theta_1^\ast$  on the grid $(\ell \tfrac{2 \pi}{n_1})_{1 \leq  \ell < \tfrac12 n_1}$, resulting in $\theta_{1,{\rm int}}^\ast$ and then choosing  a closest grid point to $\theta_2^\ast$  on the grid $(\ell \tfrac{2 \pi}{n_2})_{1 \leq  \ell < \tfrac12 n_2}$  with the constraint that the resulting value $\theta_{2,{\rm int}}^\ast$  satisfies $\theta_{2,{\rm int}}^\ast \geq \theta_{1,{\rm int}}^\ast$. Note that for obtaining similar resolutions on $\partial B_1$ and $\partial B_2$ it is natural to take $n_2 \approx n_1 R$. 
\end{remark}
\ \\[1ex]
We recall the interpolation operator in the Fourier space. We consider $\partial B_1$ with an even number $n_1$ of interpolation points. For the Fourier space we take cf. \eqref{DefUN1},
\[
  \tilde U_N^1:= U_N^1 \cup {\rm span}\{ \cos(\tfrac12 n_1 \cdot)\}, \quad N:=\tfrac12 n_1-1. 
\]
For  given $f \in C(\partial B_1)$ we define $I_N^1f \in  \tilde U_N^1$ by the interpolation property
\begin{equation}
   (I_N^1f)(X_\ell)= f(X_\ell) ,\quad \text{for all}~X_\ell \in G_{1,n_1}.
\end{equation}
The operator $I_N^2$ is defined similarly for the grid $G_{2,n_2}$ and with $N:=\tfrac12 n_2-1$. Note that there is some abuse of notation because the $N$ values in $I_N^1$ and $I_N^2$ are not necessarily the same. A modified  version of \eqref{DDschwarzinexact} is obtained by replacing the projection $P_N^i$ by the interpolation $I_N^i$, $i=1,2$. 
Given $u_i^0 \in C(\overline{B}_i)$, with ${u_i^0}_{|\partial B_i \setminus \Gamma_i}=g_i$, $i=1,2$, we determine for $n\geq 1$:\\
\begin{minipage}{0.5\textwidth}
 \begin{align*}
  \Delta u_1^n &=0\quad \text{in}~B_1\\
   u_1^n&=I_N^1(g_1,u_2^{n-1}) \quad \text{on}~\partial B_1 
 \end{align*}
\end{minipage}
\begin{minipage}{0.49\textwidth}
 \begin{equation} \label{DDschwarzinterpolate} \begin{split}
  \Delta u_2^n &=0\quad \text{in}~B_2\\
   u_2^n&=I_N^2(g_2,u_1^{n-1}) \quad \text{on}~\partial B_2 .
 \end{split} \end{equation}
\end{minipage}
\ \\[1ex]
There are two important differences between this algorithm and \eqref{DDschwarzinexact}. The first one is that by using FFT, the interpolation $I_N^if$ can be computed with significantly lower computational costs than $P_N^if$. The second difference is that, due to the interpolation of the given data $g$ at the intersection points $z_1$, $z_2$, the error is zero at these  points. \\
For the convergence analysis of algorithm~\ref{DDschwarzinterpolate} we proceed as in  Section~\ref{SectSchwarzF}, cf. \eqref{defL1N}-\eqref{iterationinexact}. It follows that the operator $L_N= \begin{pmatrix} 0 & L_{1,N}\\ L_{2,N} & 0 \end{pmatrix}$ with
\begin{equation} \label{defL1Na}
 L_{1,N}v:=R_{\overline{\Gamma}_2}\Delta_1^{-1}I_N^1(0,v),
 \end{equation}
 and $L_{2,N}$ defined similarly, determines the convergence of the algorithm. Due to the data interpolation at $z_1$, $z_2$, we now have the property that $v$ in \eqref{defL1Na} is zero on $\partial \Gamma_1$, i.e., $L_{1,N}: \, C_0(\Gamma_1) \to C^\infty(\overline \Gamma_2)$, with $C_0(\Gamma_1)$ the space of continuous functions on $\Gamma_1$ that are zero at the endpoints of $\Gamma_1$. \\
 We now study the operator $L_{1,N}$ (the results we obtain directly apply to $L_{2,N}$, too). We can not use the approach used in Section~\ref{sectNewmaximum} because there is no natural ``shift'' of the interpolation operator to the kernel as in Lemma~\ref{shiftPN}. For the interpolation operator, however, we have an explicit finite dimensional representation that gives computational access to the operator $L_{1,N}$, as we will see below.\\
 We recall some basic results for the Fourier interpolation operator $I_N^1$. For $w \in C(\partial B_1)$ let $\bw:=\big(w(X_0), \ldots w(X_{n_1-1})\big)^T \in \R^{n_1}$ be the given data vector and $x_\ell:= \ell \frac{2\pi}{n_1}$. The interpolation is given by
 \begin{align*}
   (I_N^1w)(\theta)& =\tfrac12 a_0(\bw) + \sum_{j=1}^{\frac12 n_1-1} \big( a_j(\bw) \cos (j \theta) + b_j(\bw)\sin (j \theta)\big) + \tfrac12  a_{\frac12 n_1}(\bw) \cos(\tfrac12 n_1 \theta),\\
    a_{j-1}(w) &= (\bw^TC)_j, ~1 \leq j \leq \tfrac12 n_1+1, ~ C \in \R^{n_1 \times (\frac12 n_1+1)}, ~ C_{ij}=\tfrac{2}{n_1} \cos ((j-1) x_{i-1}), \\
    b_j(\bw)&=(\bw^TS)_j,~1 \leq j \leq \tfrac12 n_1-1, ~S \in \R^{n_1 \times (\frac12 n_1-1)}, ~ S_{ij}=\tfrac{2}{n_1} \sin (j x_{i-1}).
 \end{align*}
We introduce the vectors
\[
  \bc(\theta,r):= \begin{pmatrix} \tfrac12 \\ r \cos (\theta) \\
                   r^2 \cos(2 \theta) \\
                   \vdots \\
                   r^{\frac12 n_1-1} \cos((\frac12 n_1-1)\theta) \\ 
                   \frac12 r^{\frac12 n_1}\cos(\frac12 n_1 \theta) 
                  \end{pmatrix},\quad 
                  \bs(\theta,r):= \begin{pmatrix} r \sin (\theta) \\
                   r^2 \sin(2 \theta) \\
                   \vdots \\
                   r^{\frac12 n_1-1}\sin((\frac12 n_1-1) \theta) 
                  \end{pmatrix} .
\]
\begin{lemma} \label{lemDiscrete}
The following holds for $w \in C(\partial B_1)$:
\[ 
 \big|(\Delta_1^{-1} I_N^1w)(\theta,r)\big| \leq \|C \bc(\theta,r)+ S\bs (\theta,r)\|_1  \|w\|_{L^\infty(\partial B_1)}, \quad (\theta,r)\in B_1.
 \]
\end{lemma}
\begin{proof}
For $(\theta,r) \in B_1$ we have:
\begin{align} 
 & (\Delta_1^{-1} I_N^1w)(\theta,r) \nonumber \\ & =
 \tfrac12 a_0(\bw) + \sum_{j=1}^{\frac12 n_1-1} \big( a_j(\bw) \cos (j \theta) + b_j(\bw)\sin (j \theta)\big)r^j + \tfrac12  a_{\frac12 n_1}(\bw) \cos(\tfrac12 n_1 \theta)r^{\frac12 n_1} \nonumber \\
 & = \sum_{j=1}^{\frac12 n_1+1} a_{j-1}(\bw)\bc(\theta,r)_j + \sum_{j=1}^{\frac12 n_1-1} b_j(\bw)   \bs(\theta,r)_j \label{R5} \\
  & =   \sum_{j=1}^{\frac12 n_1+1} (\bw^TC)_j\bc(\theta,r)_j + \sum_{j=1}^{\frac12 n_1-1} (\bw^TS)_j = \bw^T \big( C \bc(\theta,r)+ S\bs (\theta,r)\big).   \nonumber
\end{align}
From this we obtain the estimate
\begin{equation} \begin{split}
 \big|(\Delta_1^{-1} I_N^1w)(\theta,r)\big|  & \leq \|C \bc(\theta,r)+ S\bs (\theta,r)\|_1\|\bw\|_\infty  \\ & \leq \|C \bc(\theta,r)+ S\bs (\theta,r)\|_1  \|w\|_{L^\infty(\partial B_1)}, \quad (\theta,r)\in B_1.
\end{split} \end{equation}
\end{proof}
\ \\[1ex]
A plot of the function $(\theta,r) \to  \|C \bc(\theta,r)+ S\bs (\theta,r)\|_1$ is shown in Fig.~\ref{FigInterpol}.
\begin{figure}[ht!]
	\centering\includegraphics[width=0.45\textwidth]{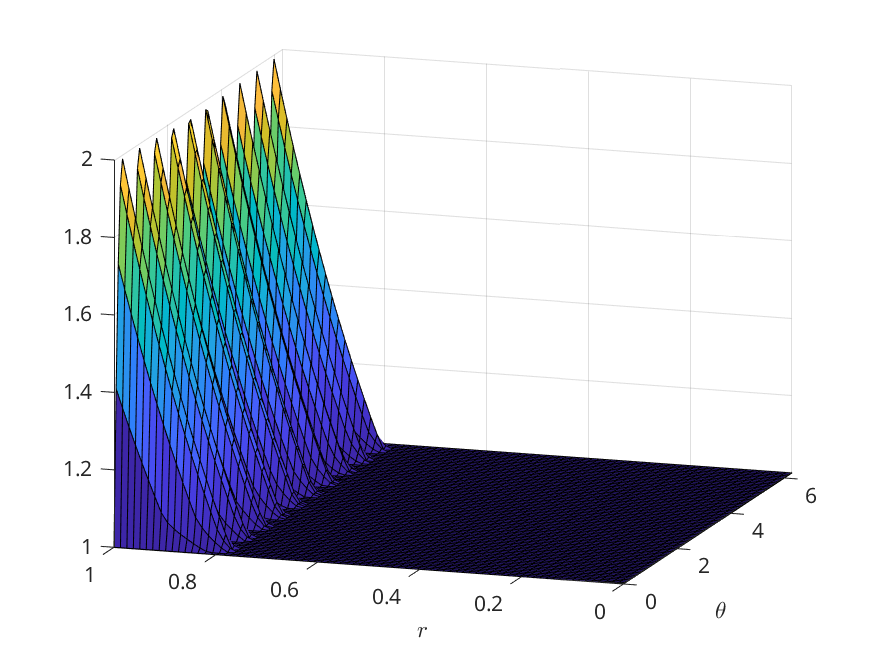}\qquad 
	\includegraphics[width=0.45\textwidth]{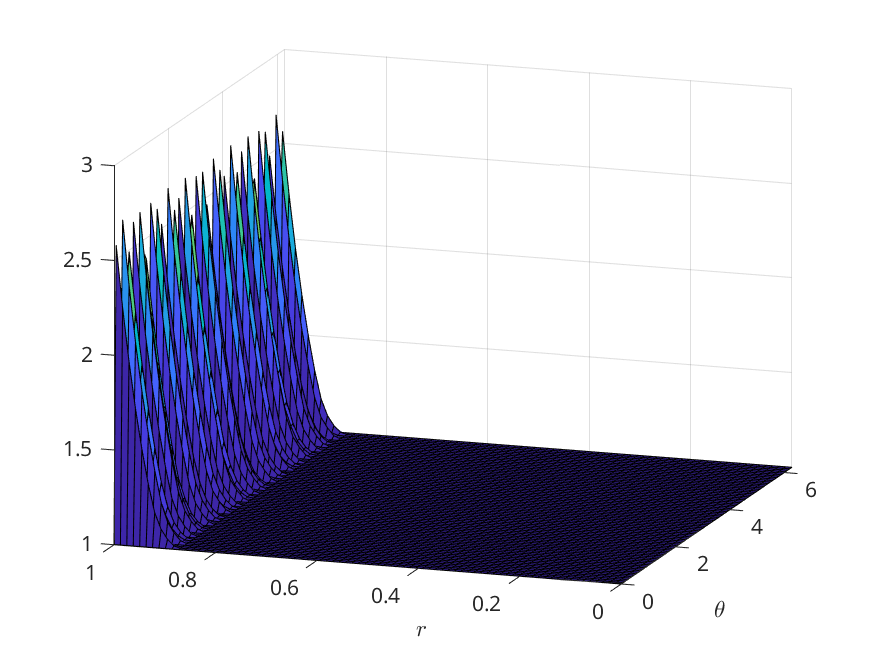}
	\caption{$\|C \bc(\theta,r)+ S\bs (\theta,r)\|_1$ for $n_1=10$ (left) and $n_1=30$ (right). \label{FigInterpol}}
\end{figure}

These results strongly indicate that also for this case, with interpolation instead of projection, we have a maximum principle as derived in Section~\ref{sectNewmaximum}, cf. \eqref{maxvariant}.
\\
For $\bw =(w(X_0), \ldots, w(X_{n_1-1}))^T$ we  introduce the projection $Q\bw \in \R^{n_1}$ with $(Q\bw)_j=w(X_{j-1})$ if $X_{j-1} \in {\rm int}(\Gamma_1)$, $1 \leq j \leq  n_1$, and  zero otherwise. Using \eqref{R5} we now study $L_{1,N}v=R_{\overline{\Gamma}_2}\Delta_1^{-1}I_N^1(0,v)$ for $v \in C_0(\Gamma_1)$. The interpolation of $(0,v)$ is given by
\[
  ((0,v)(X_0), \ldots, (0,v)(X_{n_1-1}))^T=: \bv \in \R^{n_1}.
\]
Due due $v \in C_0(\Gamma_1)$ we have $\bv=Q\bv$ and from \eqref{R5} we obtain
\begin{equation} \label{B5}
  R_{\overline{\Gamma}_2} \Delta_1^{-1} I_N^1 (0,v)= R_{\overline{\Gamma}_2} \bv^T Q\big( C \bc(\theta,r)+ S\bs (\theta,r)\big),
\end{equation}
which yields
\begin{equation} \label{contraction}
      \|L_{1,N}\|_\infty \leq \max_{(\theta,r) \in \overline{\Gamma}_2} \|Q\big( C \bc(\theta,r)+ S\bs (\theta,r)\big)\|_1.   
\end{equation}
\begin{example} \label{Ex6} \rm We determine the term $\|Q\big( C \bc(\theta,r)+ S\bs (\theta,r)\big)\|_1$ for $(\theta,r) \in \overline{\Gamma}_2$, cf. \eqref{contraction}, for specific cases, namely those considered in Example~\ref{ex1}.  First we study the case $m=1.4$, $R=1.2$ and $N=20$. For $N=20$ we have a corresponding $n_1=2(N+1)=42$.  For the grid on $\partial B_2$ we take $n_2:=50 \approx R n_1$. The $(\theta_1^\ast,\theta_2^\ast)$ values corresponding to $m=1.4$, $R=1.2$ are 
$(\theta_1^\ast, \theta_2^\ast)=(0.997,2.37)$. We choose nearby values $(\theta_{1,{\rm int}}^\ast, \theta_{2,{\rm int}}^\ast)=(1.05,2.39)$ on the grids $\big(\ell\frac{2 \pi}{n_1}\big)_{1 \leq \ell < \frac12 n_1}$ and $\big(\ell\frac{2 \pi}{n_2}\big)_{1 \leq \ell < \frac12 n_2}$.  We take the geometry corresponding to this scenario. Note that $\frac{\theta_{2,{\rm int}}^\ast-  \theta_{1, {\rm int}}^\ast}{\pi}=0.44$. We evaluate $\|Q\big( C \bc(\theta,r)+ S\bs (\theta,r)\big)\|_1$ for $(\theta,r)\in \Gamma_2$. As in Example~\ref{ex1}, we parameterize $\Gamma_2$ by using the angle coordinate $\tilde \theta$ in $B_2$, with $\tilde \theta \in [\theta_{2,{\rm int}}^\ast, 2 \pi -\theta_{2,{\rm int}}^\ast]$. We repeat this for $N=40$ with corresponding values $n_1=82$, $n_2=98$, $(\theta_{1,{\rm int}}^\ast, \theta_{2,{\rm int}}^\ast)=(0.996,2.37) $. The results are shown in Fig.~\ref{FigExpIN} (left). In Example~\ref{ex1}  we also considered the case  $m=2.1$, $R=1.2$  (small overlap), which has corresponding angle values $(\theta_1^\ast,\theta_2^\ast)=(0.333,2.87)$ and  contraction number $C_1(\theta_1^\ast,\theta_2^\ast)=0.807$. For this case we also perform the same experiment  for $N=20$ (with $(n_1,n_2)=(42,50)$, $(\theta_{1,{\rm int}}^\ast, \theta_{2,{\rm int}}^\ast)=(0.299,2.89)$) and $N=40$ (with  $(n_1,n_2)=(82,98)$, $(\theta_{1,{\rm int}}^\ast, \theta_{2,{\rm int}}^\ast)=(0.307,2.89)$). The results are shown in Fig.~\ref{FigExpIN} (right). 
\begin{figure}[ht!]
	\centering\includegraphics[width=0.45\textwidth]{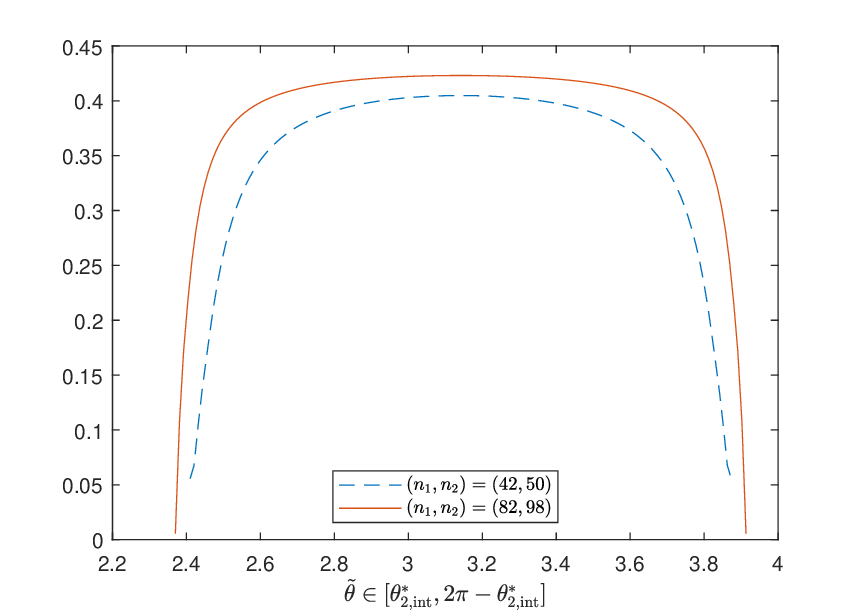}\qquad 
	\includegraphics[width=0.45\textwidth]{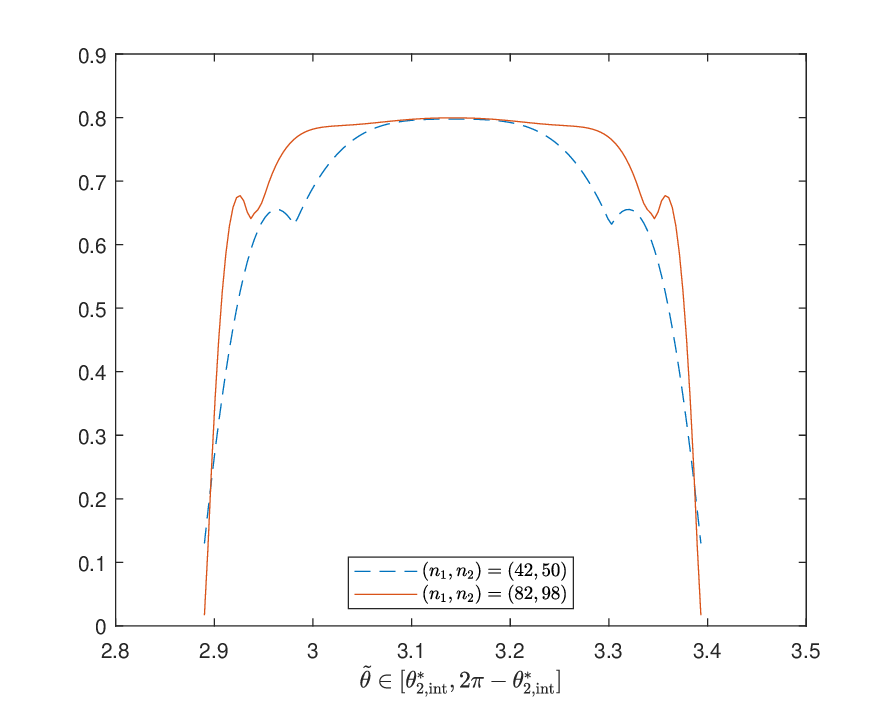}
	\caption{$\|Q\big( C \bc(\theta,r)+ S\bs (\theta,r)\big)\|_1$, evaluated on  $\Gamma_2$, for $m=1.4$, $R=1.2$, and $N=20$, $N=40$ (left), and $m=2.1$, $R=1.2$, and $N=20$, $N=40$ (right). \label{FigExpIN}}
\end{figure}

Comparing these results with the ones shown in Fig.~\ref{FigExpPN} we note the following. Firstly, in Fig.~\ref{FigExpIN} we observe similar nearly constant plateaus as in Fig.~\ref{FigExpPN}. Also the values  attained
at these plateaus are similar and have size approximately $C_1(\theta_1^\ast,\theta_2^\ast)$. Secondly, also in Fig.~\ref{FigExpIN} we observe boundary layers with a width that decreases as a function of $N$. An important difference  between the results in Fig.~\ref{FigExpPN} and in Fig.~\ref{FigExpIN} is that in the former the endpoint values are close to $0.5$, whereas in the latter the endpoint values are 0 (per construction). Finally note that the oscillations observed in the right panel of Fig.~\ref{FigExpIN} are related to the Gibbs phenomenon. 
\end{example}
\begin{example}\label{ex62}\rm  The Gibbs phenomenon turns out to have a strong effect if we consider a case wit a (very) large overlap. To illustrate this we perform an experiment for the case  $R=1.7$, $m=0.75$,
$(\theta_{1}^\ast,\theta_{2}^\ast)=(2.66,2.86)$, $C_1(\theta_{1}^\ast,\theta_{2}^\ast)=0.064$. For $N=20$ we take $(n_1,n_2)=(42,72)$ and interpolated angle values $(\theta_{1,{\rm int}}^\ast,\theta_{2,{\rm int}}^\ast)=(2.69,2.88)$. For $N=40$: $(n_1,n_2)=(82,140)$, $(\theta_{1,{\rm int}}^\ast,\theta_{2,{\rm int}}^\ast)=(2.68,2.87)$. Results are shown in Fig.~\ref{FigExpIN2}. The $\times$-markers correspond to the  values at the grid points $\Gamma_2 \cap G_{2,n_2}$. Note that these values due \emph{not} suffer from the oscillations due to the Gibbs phenomenon.  
\begin{figure}[ht!]
	\centering\includegraphics[width=0.45\textwidth]{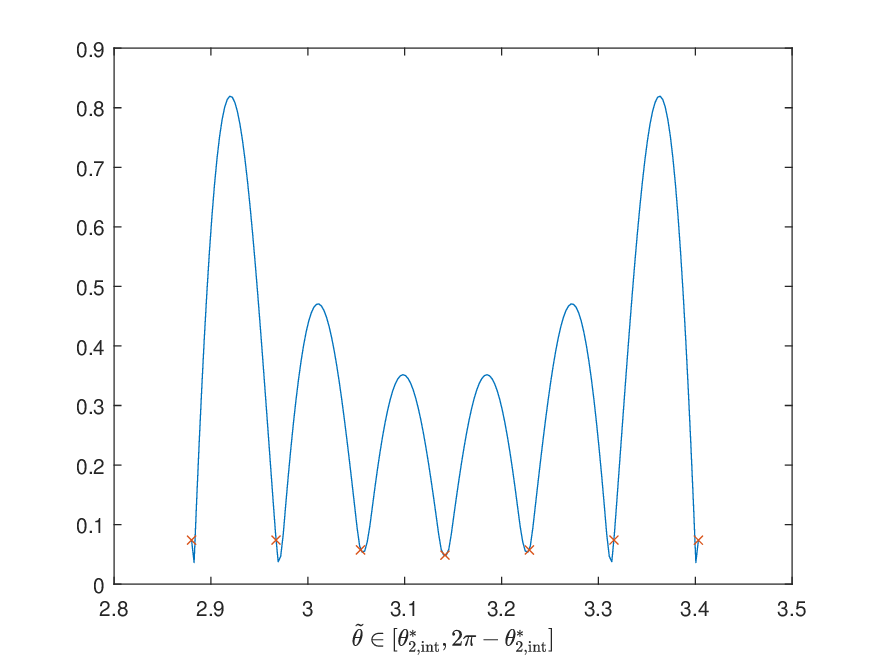}\qquad 
	\includegraphics[width=0.45\textwidth]{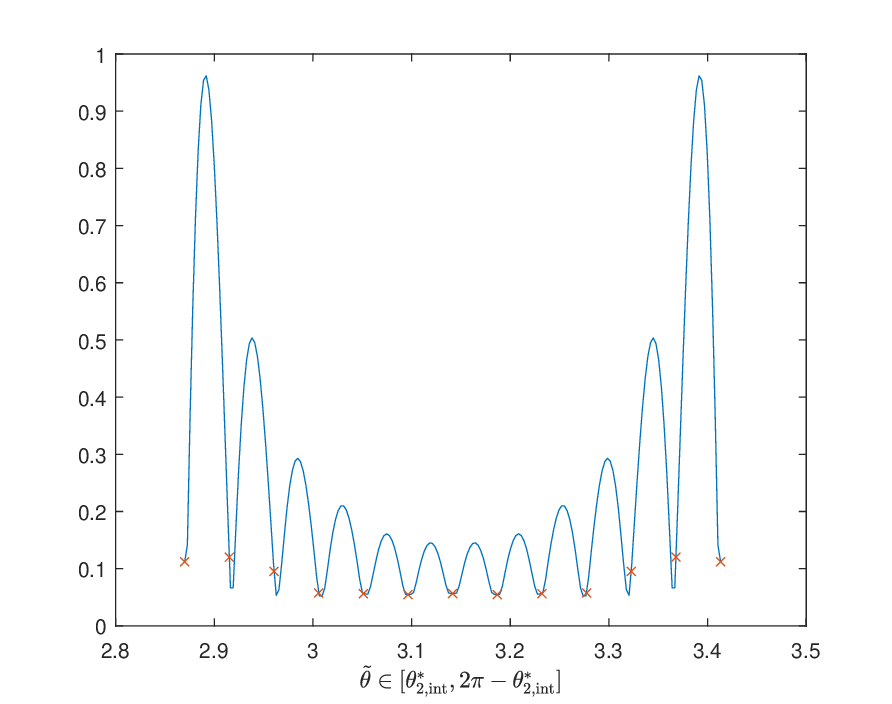}
	\caption{$\|Q\big( C \bc(\theta,r)+ S\bs (\theta,r)\big)\|_1$, evaluated on  $\Gamma_2$, for $R=1.7$, $m=0.75$, and $N=20$  (left), and $N=40$ (right). The '$\times$' are the evaluations at the grid points $G_{2,n_2}$ with $n_2$ the even number closest to $Rn_1$.  \label{FigExpIN2}}
\end{figure}
\end{example}

We now turn to a systematic  parameter study for the contraction number $\|L_{1,N}\|_\infty$. We are particularly interested in the dependence of the contraction number on the parameter $N$ (directly related to the dimension of the discrete Fourier space) and the geometric scenario (small/large overlap).
We note that in the iteration \eqref{DDschwarzinterpolate} we use values only at the interpolation points. This means that the bound in \eqref{contraction} is too pessimistic in the sense that the maximum over $ \overline{\Gamma}_2$ can be replaced by the maximum over the grid $\overline{\Gamma}_2 \cap G_{2,n_2}$. Hence we obtain the contraction number bound:
\begin{equation} \label{contraction2}
      \|L_{1,N}\|_\infty \leq \max_{(\theta,r) \in \overline{\Gamma}_2 \cap G_{2,n_2}} \|Q\big( C \bc(\theta,r)+ S\bs (\theta,r)\big)\|_1 =: C(N,\theta_1^\ast,\theta_2^\ast).
\end{equation}
We first consider  the case of two discs with the same radius. This means that we restrict to $\theta_1^\ast \in [0,\tfrac{\pi}{2})$ and $\theta_2^\ast=\pi- \theta_1^\ast$. For $\theta_1^\ast$ close to $0$ one has a small overlap and for $\theta_1^\ast$ close to $\tfrac{\pi}{2}$ one has a large overlap. For these symmetric scenarios we have $C_1^s(\theta_1^\ast):=C_1(\theta_1^\ast,\pi- \theta_1^\ast)= 1 - \tfrac{2 \theta_1^\ast}{\pi}$ as contraction number bound of the Schwarz method applied to the continuous problem, cf. Theorem~\ref{thmconstraction}. The bound in \eqref{contraction2} is denoted by $C(N,\theta_1^\ast):=C(N,\theta_1^\ast,\pi- \theta_1^\ast)$.  Since the discs have the same radius we take $n_2=n_1=2(N+1)$. We use the same  approach as in Example~\ref{Ex6}: for  given $\theta_1^\ast$ and $N$ we determine $\theta_{1,{\rm int}}^\ast$ and $ \theta_{2,{\rm int}}^\ast:=\pi-\theta_{1,{\rm int}}^\ast$ and compute the bound \eqref{contraction2} with $\theta_i^\ast$ replaced by $\theta_{i, {\rm int}}^\ast$. We take $N \in \{10,20,40,140\}$ and for given $N$ we vary $\theta_1^\ast \in (0,\tfrac{\pi}{2})$.  Results are shown in Fig.~\ref{FigContraction} (left).  We repeat the experiment but now with $B_2$ a ball of radius $1.7$ and $n_2\approx 1.7 n_1$. One then has to vary $\theta_1^\ast \in (0,\pi)$ (from small overlap to large overlap). In that case we have $\theta_2^\ast=\pi- \arcsin (\tfrac{1}{1.7}\sin(\theta_1^\ast))$ and $\tilde C(\theta_1^\ast):=C_1(\theta_1^\ast, \theta_2^\ast)=1-\tfrac{1}{\pi}\big( \arcsin (\tfrac{1}{1.7}\sin(\theta_1^\ast)) +\theta_1^\ast\big)$. Results are shown in Fig.~\ref{FigContraction} (right).
\begin{figure}[ht!]
	\centering\includegraphics[width=0.45\textwidth]{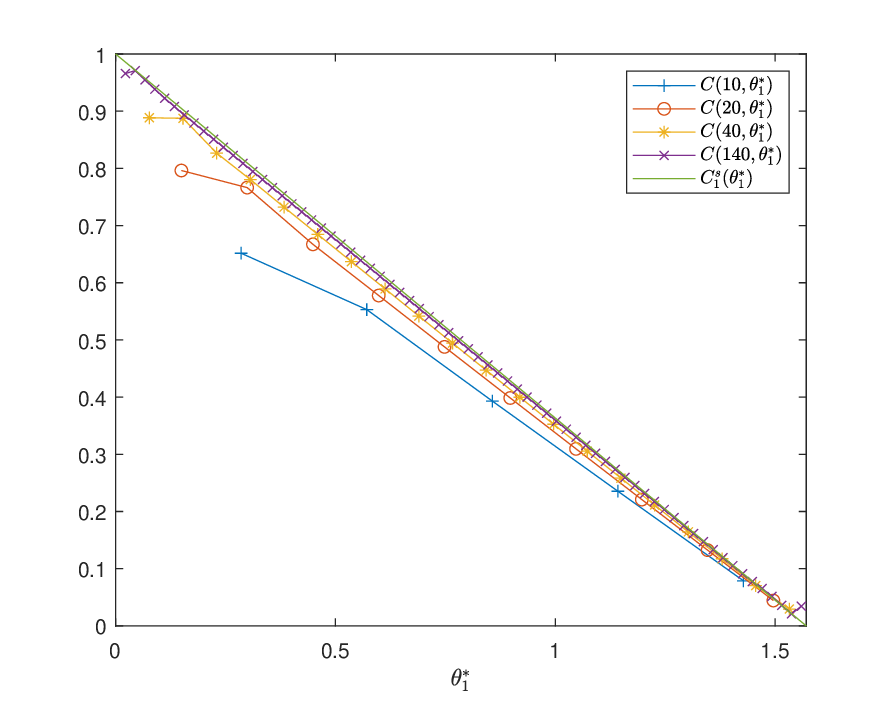}\qquad 
	\includegraphics[width=0.45\textwidth]{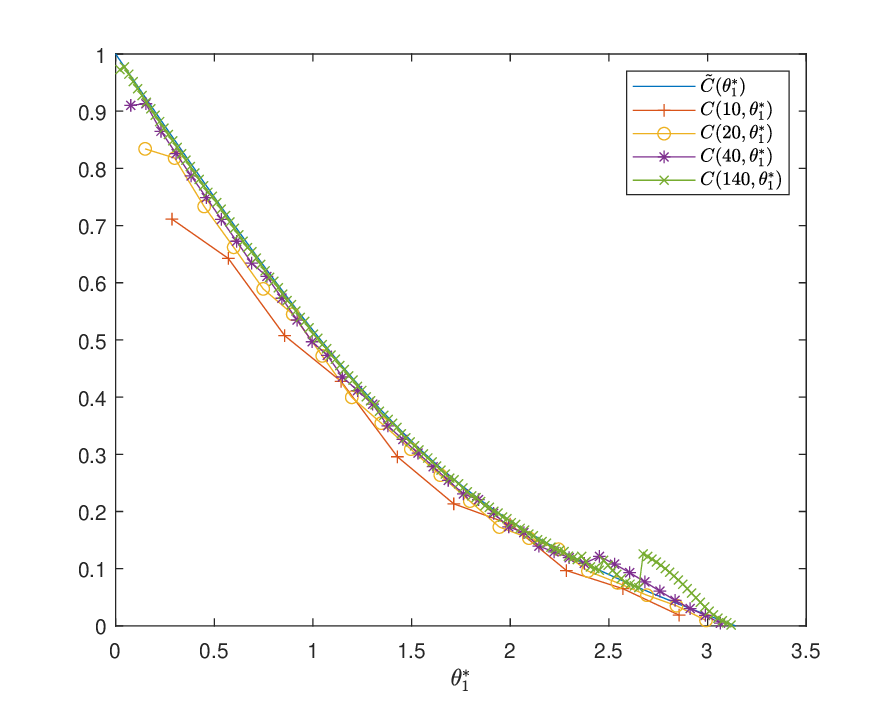}
	\caption{Contraction number bound \eqref{contraction} for $R=1$ and $\theta_{1}^\ast \in (0,\tfrac{\pi}{2})$ (left) and $R=1.7$ and $\theta_{1}^\ast \in (0,\pi)$ (right), with $n_2 \approx R n_1$. \label{FigContraction}}
\end{figure}

We observe that in both examples the values of the contraction number bound $C(N,\theta_1^\ast,\theta_2^\ast)$ are very similar to that of the contraction number bound $C_1(\theta_1^\ast,\theta_2^\ast)$ corresponding to the continuous case. This in particular also holds for small $N$ values. \\
The somewhat irregular behaviour close to $\theta_1^\ast=2.5$ is the right panel in Fig.~\ref{FigContraction} is due to some irregularity in the geometry interpolation procedure $(\theta_1^\ast,\theta_2^\ast) \to (\theta_{1,{\rm int}}^\ast,\theta_{2, {\rm int}}^\ast)$.
\begin{remark} \label{RemEndpoint}
\rm We observe in Fig.~\ref{FigContraction} that the Schwarz-Fourier method with interpolation has, also for small $N$ values, a maximum norm contraction number that tends to $0$ if the geometric scenario tends to a complete overlap of $B_1$ by $B_2$. Such a result can not hold for the Schwarz-Fourier method with projection, cf. Sections~\ref{SectSchwarzF}-\ref{sectNewmaximum}. The reason for this is the behaviour of $R_{\overline{\Gamma}_2}\Delta_1^{-1}P_N^1 \chi_{\Gamma_1}$ in the boundary layers, cf. Fig.~\ref{FigExpPN}. We have (as expected) $(R_{\overline{\Gamma}_2}\Delta_1^{-1}P_N^1 \chi_{\Gamma_1})(z_i) \approx \tfrac12$ and therefore $\tfrac12$ is (close to) a lower bound for boundary layer part $\|R_{\overline{\Gamma}_2}\Delta_1^{-1}P_N^1 \chi_{\Gamma_1}\|_{L^\infty(\overline{\Gamma}_2 \setminus \Gamma_2^{(N)})}$ of the contraction number, cf. \eqref{esti}. 
\end{remark}
\ \\[1ex]
Finally we note that, as can already be seen from the results presented in Fig.~\ref{FigExpIN2}, for having small contraction numbers for larger $\theta_1^\ast$ values, cf.~Fig.~\ref{FigContraction}, it is essential that we take $n_2 \approx Rn_1$, i.e., the grid sizes on the two circles are approximately the same. To illustrate this, for the case $R=1.7$  we perform the same experiment as above but now with $n_2 \approx 1.8 R n_1$ instead of $n_2\approx R n_1$. Results are presented in Fig.~\ref{FigContractionbad}.
\begin{figure}[ht!]
	\centering\includegraphics[width=0.6\textwidth]{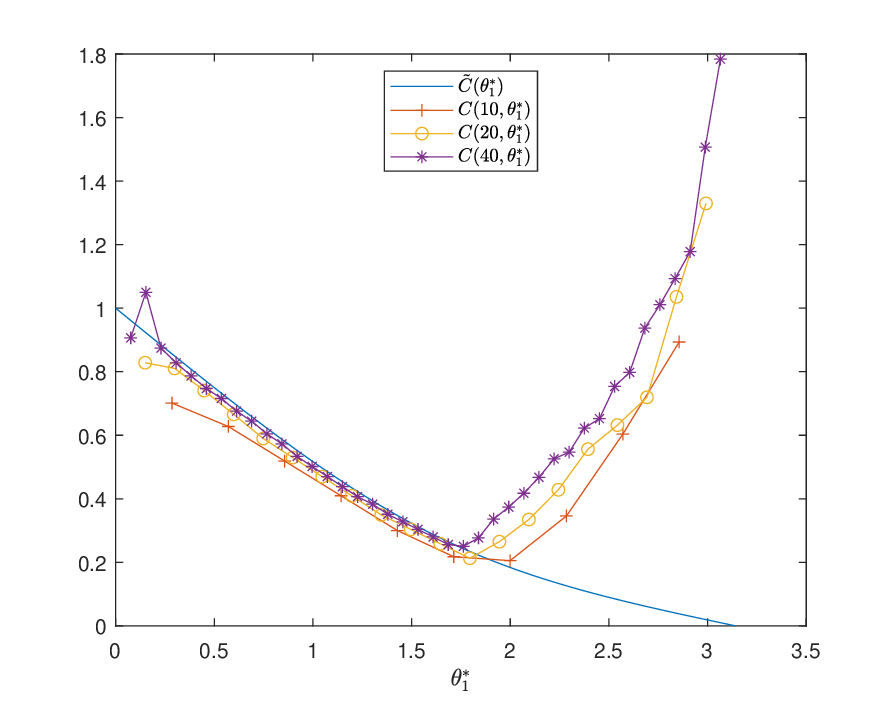}
	\caption{Contraction number bound \eqref{contraction} for  $R=1.7$ and $\theta_{1}^\ast \in (0,\pi)$, with $n_2 \approx 1.8 R n_1$. \label{FigContractionbad}}
\end{figure}

\noindent
{\bf Acknowledgements}\\
The author  acknowledges funding by the Deutsche
Forschungsgemeinschaft (DFG, German Research Foundation) - Project number 442047500 through
the Collaborative Research Center "Sparsity and Singular Structures" (SFB 1481).

\section{Appendix}
We give a proof of the result in Lemma~\ref{lemdisc}. Define $\delta:=g_0^+-g_0^-$, $\ell(\theta):=\frac{\delta}{2 \pi} \theta$ and $g_c(\theta):=g(\theta)+\ell(\theta)$. The function $g_c$ is continuous at $0$ with value $g_c(0)=g_0^+$. Using this continuity property and \eqref{ResK} we obtain
\begin{equation} \label{A1} \begin{split} 
 \lim_{r \uparrow 1} u\big(\theta(r),r\big) & = \lim_{r \uparrow 1}\frac{1}{2 \pi} \int_0^{2 \pi} K\big(\theta(r)-\theta',r\big) \big(g_c(\theta') -\ell(\theta')\big) \, d\theta'\\
 & = g_0^+ -   \lim_{r \uparrow 1}\frac{1}{2 \pi} \int_0^{2 \pi} K\big(\theta(r)-\theta',r\big) \ell(\theta') \, d\theta'.
\end{split}
 \end{equation}
 We analyze the convolution of the linear function $\ell$ with $K$. Using \eqref{Poissonkernel} and partial integration and with the notation $\theta=\theta(r)$ we get  
 \begin{align*}
  \frac{1}{2 \pi} \int_0^{2 \pi} K\big(\theta(r)-\theta',r\big) \ell(\theta') \, d\theta' &= \frac{\delta}{(2 \pi)^2} \int_0^{2 \pi} \big( 1+ 2 \sum_{n=1}^\infty \cos\big(n(\theta-\theta')\big) r^n\big) \theta' \, d\theta' \\ &= \tfrac12 \delta  + \frac{\delta}{2 \pi^2}\sum_{n=1}^\infty r^n \int_0^{2 \pi}  \cos\big(n(\theta-\theta')\big)  \theta' \, d\theta' \\
  &=  \tfrac12 \delta - \frac{\delta}{2 \pi^2}\sum_{n=1}^\infty r^n  \frac{1}{n}  \sin\big(n(\theta-\theta')\big) \theta' \Big\rvert_{0}^{2 \pi} + 0 \\
  &= \tfrac12 \delta - \frac{\delta}{\pi}\sum_{n=1}^\infty   \frac{r^n}{n}  \sin(n\theta).  
 \end{align*}
Combining this with \eqref{A1} we obtain
\begin{equation} \label{A2} 
 \lim_{r \uparrow 1} u\big(\theta(r),r\big) = \tfrac12 (g_0^++g_0^-) +\frac{\delta}{\pi}\lim_{r \uparrow 1}\sum_{n=1}^\infty   \frac{r^n}{n}  \sin\big(n\theta(r)\big).  
\end{equation}
We further analyze the series $\sum_{n=1}^\infty   \frac{r^n}{n}  \sin(n\theta)$, $\theta=\theta(r)$:
\begin{equation} \label{A3} \begin{split}
 \sum_{n=1}^\infty   \frac{r^n}{n}  \sin(n\theta) &=  \sum_{n=1}^\infty   \frac{1}{n} \frac{1}{2i}\big(e^{in\theta}-e^{-i n \theta}\big) r^n \\
 &  =  \frac{1}{2i}\sum_{n=1}^\infty   \frac{1}{n}\left(\Big(re^{in\theta}\Big)^n-\Big(r e^{-i n \theta}\Big)^n\right) \\ & = \frac{1}{2i} \big( \ln(1-re^{-i\theta})-\ln(1-re^{i\theta}) \big).
 \end{split}
\end{equation}
Note $1-re^{-i\theta}= 1-r \cos\theta +i r \sin\theta= \tilde r e^{i \tau}$ with $\tilde r= (1-2 r \cos \theta +r^2)^\frac12$, $\tau= {\rm sign}(\theta) \arccos \big( \frac{1-r \cos \theta}{\tilde r}\big)$. This yields
\begin{align*}
 \frac{1}{2i} \big( \ln(1-re^{-i\theta})-\ln(1-re^{i\theta}) \big) &= \frac{1}{2i} \big( \ln(\tilde r e^{i \tau})-\ln(\tilde re^{-i\tau}) \big) \\
 & = \frac{1}{2i} \big( \ln \tilde r + i \tau - \ln \tilde r + i \tau \big) =\tau. 
\end{align*}
Combining  this with \eqref{A2}- \eqref{A3} we obtain
\begin{equation} \label{A4}
\lim_{r \uparrow 1} u\big(\theta(r),r\big) = \tfrac12 (g_0^++g_0^-) +\frac{\delta}{\pi}{\rm sign}(\theta)\lim_{r \uparrow 1}\arccos \Big( \frac{1-r \cos \big( \theta(r)\big)}{\big(1-2 r \cos \big(\theta(r)\big) +r^2\big)^\frac12}\Big).
\end{equation}
With $\theta=\theta(r)$ we have
\begin{equation} \label{A5} \begin{split}
 \frac{1-r \cos \theta}{(1-2 r \cos\theta +r^2)^\frac12}& =  \frac{1- \cos \theta + (1-r) \cos \theta}{((1-r)^2 + 2 r (1-\cos\theta))^\frac12} \\
 & = \frac{ 2 \sin^2(\tfrac12 \theta) + (1-r) \cos \theta}{((1-r)^2 + 4 r \sin^2(\tfrac12 \theta))^\frac12}\\
 &= \frac{ \cos \theta + 2 \sin(\tfrac12 \theta) f(\theta,r)}{\big(1+ 4 r f(\theta,r)^2\big)^\frac12},
\end{split}
\end{equation}
with $f(\theta,r):= \frac{\sin(\tfrac12 \theta)}{1-r}$. Now note that $\sin\big(\tfrac12 \theta(1)\big)=0$ and from a linear Taylor expansion at $r=1$ we obtain $ \sin\big(\tfrac12 \theta(r)\big)=\tfrac12 (r-1)\cos\big(\theta(\xi_r)\big) \theta'(\xi_r)$, with $r \leq \xi_r <1$. This implies $\lim_{r \uparrow 1} f(\theta(r),r)=- \tfrac12 \theta'(1)$. Using this in \eqref{A5} yields
\[
 \lim_{r \uparrow 1}\arccos \Big( \frac{1-r \cos \big( \theta(r)\big)}{\big(1-2 r \cos \big(\theta(r)\big) +r^2\big)^\frac12}\Big) = \arccos \big( \big( 1+ \theta'(1)^2\big)^{-\frac12}\big),
\]
and combining this with \eqref{A4} completes the proof. 

\bibliographystyle{siam}
\bibliography{literatur}{}

\end{document}